\documentclass[11pt]{amsart}
\usepackage{fullpage}

\usepackage{hyperref}

\usepackage{amsmath,amsfonts,amssymb}
\usepackage{mathrsfs}
\usepackage{verbatim}
\usepackage{amsthm}
\usepackage{mathrsfs}

\newtheorem{theorem}{Theorem}[section]
 \newtheorem{corollary}[theorem]{Corollary}
 \newtheorem{lemma}[theorem]{Lemma}
 \newtheorem{proposition}[theorem]{Proposition}
 \theoremstyle{definition}
 \newtheorem{definition}[theorem]{Definition}
 \theoremstyle{remark}

 \numberwithin{equation}{section}

\def \bC {\mathbb C}

\def \bN {\mathbb N}

\def \bQ {\mathbb Q}
\def \bR {\mathbb R}
\def \bS {\mathbb S}
\def \bT {\mathbb T}

\def \bZ {\mathbb Z}

\def \cB {\mathcal B}

\def \cD {\mathcal D}

\def \cF {\mathcal F}

\def \cH {\mathcal H}

\def \cL {\mathcal L}

\def \cO {\mathcal O}

\def \cS {\mathcal S}

\def \fg {\mathfrak g}

\def \ft {\mathfrak t}

\def \fz {\mathfrak z}

\def\Gh{{\widehat{G}}}
\def \sL{\mathscr L}

\def \RepG {\text{\rm Rep}(G)}
\def \FundG {\text{\rm Fund}(G)}

\def \id {\text{\rm I}}
\def\Op{{{\rm Op}}}
\def\rank{{{\rm rank}}}

\def\L2f{L^2_{\mbox{\tiny finite}}(G)}
\def \RepG {\text{\rm Rep}(G)}
\def \FundG {\text{\rm Fund}(G)}

\def \jac {\text{\rm jac}}

\def \sL{\mathscr L}
\def \tr {\text{\rm tr}}
\def \TR {\text{\rm TR}}
\def \id {\text{\rm I}}
\def\Op{{{\rm Op}}}
\def\supp{{{\rm supp}}}
\def\res{{\text{\rm res}}}
\def \sig {\varsigma}

\def \spec {\text{\rm Spec}}

\begin{document}

\title
{Local and global symbols on compact Lie groups}
\author
{V\'eronique Fischer}
\address
{Department of Mathematical Sciences,
University of Bath,
Claverton Down,
Bath  BA2 7AY, United Kingdom}
\email{v.c.m.fischer@bath.ac.uk}



\date{October 2018}

\maketitle

\begin{abstract}
On the torus, it is possible to assign 
a global symbol to a  pseudodifferential operator using Fourier series. 
In this paper we investigate the relations 
 between the local and global symbols 
for the operators in the classical H\"ormander calculus
and describe the principal symbols, the non-commutative residue and the canonical trace of an operator in terms of its global symbol.
We also generalise these results to any compact Lie group.  
\end{abstract}

\tableofcontents

\section{Introduction}

Pseudo-differential operators on a compact manifold are by definition 
the operators which in local coordinates are pseudo-differential,
i.e. belong to some H\"ormander class $\Psi^m$.
Although this definition is local, each operator remains globally defined but with local symbols in coordinate charts.
On the torus, and more generally on a compact Lie group, it is possible to use the Fourier series to define a global symbol for these operators.
The aim of this paper is to characterise the 
classical H\"ormander classes in terms of the global symbols 
and then to
investigate the relations between 
the global symbols and the symbols in any coordinate. 
In particular, we will relate the global symbol with the principal symbol, the non-commutative residue 
and the canonical trace.
The definitions of the non-commutative residue and of the canonical trace together with further references will be given in Section \ref{sec_Rn}, so in this introduction  
we will restrict our comments to their origins and uses.
The non-commutative residue was introduced independently by Guillemin \cite{guillemin} 
and Wodzicki \cite{Wod_84,Wod_87} in the early eighties.
Beside being the only trace on the algebra of pseudo-differential operators on $M$ up to constants, 
its importance comes from its applications in mathematical physics, 
mainly in Connes' non-commutative geometry due to its link with the Dixmier trace \cite{connes}
but also in relation with e.g. the Einstein-Hilbert action
(see \cite[Section 6.1]{lesch_2010} and the references therein).
The canonical trace was constructed by Kontsevich and Vishik in the mid-nineties \cite{KV} as a tool to study further zeta functions and determinants of elliptic pseudo-differential operators. Since then, it 
has received considerable attention and found interesting applications, 
see e.g. \cite{scott,paycha,okikiolu}.

On the torus $\bT^n$, the relations between local and global symbols are  simple using trigonometric Fourier series:
a global symbol $\sigma$ is defined on $\bT^n \times \bZ^n$, 
and, when the corresponding operator is classical,
this global symbol extends naturally to a symbol on $\bT^n\times \bR^n$  which coincides with the local symbol modulo smoothing symbols, see Section \ref{sec_T}.
Therefore the notions of principal symbol and non-commutative residue have obvious meanings locally and globally.
Furthermore, the recent result of the author of this paper on real trace expansions
\cite{myexp} applied to the Laplacian on the torus
yield a description of the non-commutative residue and of the canonical trace as coefficients in certains expansions.
This generalises Pietsch's result \cite[Theorem 11.15]{pietsch}
on  the non-commutative residue.
It also gives an alternative  proof of the description of the canonical trace as the discrete finite part of the symbol already obtained by L\'evy, Jimenez and Paycha in \cite[Section 5]{Paycha+_tams}.

The case of compact Lie groups is more involved but also more 
 natural from the viewpoint of harmonic analysis. Indeed,
the idea of global symbols stems from the study of Fourier multipliers, and  is related to  singular integrals (Caledr\`on-Zygmund theory) and to the genesis of the pseudodifferential theory;
see  \cite{monarxiv}  and its introduction for a brief survey on Fourier multipliers on compact Lie groups in these directions.
The idea of studying pseudodifferential operators on Lie groups as a systematic generalisation of Fourier multipliers can be traced back to Michael Taylor in \cite{taylor_bk86}.
On compact Lie groups, recent works \cite{monJFA,ruzhansky+turunen_bk,
ruzhansky+turunen+wirth} have shown that 
it is possible to characterise the  classes of global symbols corresponding to pseudodifferential operators (see Section \ref{subsec_psiG}).
In this paper, we also show that it is also true for the classical pseudodifferential classes (see Section \ref{subsec_hom_symb}).
Not surprisingly, 
since the global symbol is built out of the representations of the group, this will involve the highest weight theory.

Although this paper is concerned with compact Lie groups, 
it naturally applies to settings such as  spheres of any dimension and more generally compact homogeneous domains
which are quotients of a compact Lie group by a closed subgroup. 

\smallskip

The paper is organised as follows.
In Section \ref{sec_Rn}, 
we review the properties of the classical pseudodifferential calculus on any compact manifold.
In particular, we recall  the notions of principal symbols, non-commutative residue and canonical trace.
 We explore the case of the torus
 in Section \ref{sec_T}, the main result in the toroidal setting being summarised in Section \ref{subsec_psiclT}.
In Section \ref{sec_G}, we examine the case of compact Lie groups: we first  relate the principal symbol of a classical operator with the global symbol in Section \ref{subsec_princ_symbol}.
Subsequently, in Section \ref{subsec_hom_symb}, we  define the notions of global homogeneous symbol  
and global classical symbols, and characterise classical pseudo-differential operators as the pseudo-differential operators with classical symbols. 
We conclude with obtaining the non-commutative residue and the canonical trace as coefficients in certains expansions.

\smallskip

\textbf{Notation}
$\bN_0=\{0,1,2,\ldots\}$ denotes the set of non-negative integers
and 
$\bN=\{1,2,\ldots\}$ the set of positive integers.

If $\cH_1$ and $\cH_2$ are two Hilbert spaces, we denote by $\sL(\cH_1,\cH_2)$ the Banach space of the bounded operators from $\cH_1$ to $\cH_2$. If $\cH_1=\cH_2=\cH$ then we write $\sL(\cH_1,\cH_2)=\sL(\cH)$.
The operator norm is then denoted by $\|\cdot\|_{\sL(\cH)}$, 
where the Hilbert-Schmidt norm is denoted by 
$\|\cdot\|_{HS(\cH)}$.
We may keep the same notation for an operator and its possible extensions to larger spaces when the extensions are unique by continuity. 

We denote by $C(M)$ the space of complex valued continuous functions on a manifold $M$, by 
$C^\infty(M)$ the subspace of smooth functions in $C(M)$
and, by $\cD(M)$ the subspace of compactly supported functions in $C^\infty(M)$. 
We will regularly  use the usual notation 
$\langle \xi \rangle = \sqrt{1+|\xi|^2}$, 
$\partial_j=\partial_{x_j}$ for the partial derivatives in $\bR^n$, 
$\partial^\alpha=\partial_1^{\alpha_1}\partial_2^{\alpha_2}\ldots$ etc.

\section{Preliminaries: the Euclidean case}
\label{sec_Rn} 
 
In this section we recall well-known properties of the classical H\"ormander pseudodifferential calculus on $\bR^n$ and on a manifold.
We also introduce the non-commutative residue and the canonical trace
together with their relations to trace expansions.

\subsection{The Euclidean pseudodifferential calculus
}
\label{subsec_prel_EPDO}

In this section we mainly set the notation and our vocabulary for the pseudodifferential calculus  on an open subset $\Omega$ of $\bR^n$
and on a compact manifold $M$.
Classical references for this material include \cite{shubin_bk,taylor_bk_princeton}.

We denote by $S^m=S^m(\Omega)$ the H\"ormander class of symbols of order $m\in \bR$ on $\Omega$, 
that is, the Fr\'echet space of smooth functions $a:\Omega\times \bR^n\to \bC$ satisfying 
for all multi-indices $\alpha,\beta\in\bN_0^n$
$$
 |\partial_x^\beta \partial_\xi^\alpha a(x,\xi)| \leq C_{\alpha,\beta}
\langle \xi \rangle^{m-\alpha}.
$$ 

We say that the symbol $a\in S^{m}(\Omega)$ is  \emph{compactly supported in $x$} when there exists $R>0$ such that 
$a(x,\xi)=0$ for any $(x,\xi)\in \bR^n\times\bR^n$ with $|x|>M$.

\medskip

To each symbol $a\in S^m(\Omega)$, 
we associate the operator $\Op_\Omega(a)$ defined via
$$
\Op_\Omega(a) f(x) = \int_{\bR^n} \widehat f(\xi) e^{2i\pi x\cdot \xi} a(x,\xi) d\xi, \qquad x\in \Omega, \ f\in \cD(\Omega).
$$
Here, $\widehat f$ denotes the Euclidean Fourier transform of $f\in \cS(\bR^n)$:
$$
 \widehat f(\xi) = \cF_{\bR^n} f(\xi) =\int_{\bR^n} f(x) e^{-2i\pi x\cdot \xi} dx.
 $$
We denote by $\Psi=\Psi^m(\Omega) = \Op_\Omega(S^m(\Omega))$   
the H\"ormander class of operators of order $m\in \bR$ on $\Omega$.
Recall that $\Op_\Omega$ is one-to-one on $S^m(\Omega)$ 
and thus that $\Psi^m(\Omega)$ inherit a structure of Fr\'echet space.   

The class of smoothing symbols is denoted by $S^{-\infty}=S^{-\infty}(\Omega)=\cap_{m\in \bR} S^m(\Omega)$ and the class of smoothing symbols is denoted by $\Psi^{-\infty}(\Omega) =\cap_{m\in \bR} \Psi^m(\Omega)=\Op_\Omega(S^{-\infty}(\Omega))$.
Examples of smoothing operators are convolution operators with Schwartz convolution kernels.

\medskip

If $A\in \cup_{m\in \bR} \Psi^m(\Omega)$ then 
we denote by $K_A \in \cS'(\Omega\times\Omega)$ its integral kernel
so that we have in the sense of distributions:
$$
Af(x)
=\int_{\Omega} K_A(x,y) f(y) dy.
$$
Recall that $a(x,\xi) = \cF_{\bR^n}( K_A(x, x- \cdot))$
and that $K_A$ is smooth away from the diagonal $x=y$. 

The following property is well-known:
\begin{lemma}
\label{lem_tr_Psim<-n}
If  $A=\Op_\Omega(a)\in \Psi^m(\Omega)$  with $m<-n$,
then its integral kernel $K_A$ is continuous.
Assuming furthermore that its symbol $a\in S^{m}(\Omega)$ is  compactly supported in $x$, 
then the operator  $A$  is trace-class with trace:
$$
\tr (A) = \int_{\Omega} K_A(x,x) dx 
 =\int_{\Omega\times\bR^{n}} a(x,\xi) \ dx d\xi.
$$
Consequently, if $\Omega$ is a bounded open subset of $\bR^n$, 
then the linear map $A\mapsto \tr (A)$ is continuous on $\Psi^m(\Omega)$ for any $m<-n$.
\end{lemma}

A symbol $a\in S^m$ with $m\in \bR$ admits a \emph{poly-homogeneous expansion with complex order} $\tilde m$ 
when it admits an expansion $a\sim \sum_{j\in \bN_0} \alpha_{m-j}$ 
where each function  $\alpha_{m-j}(x,\xi)$ is  $(\tilde m-j)$-homogeneous in $\xi$ for $|\xi|\geq 1$; here $\tilde m\in \bC$ with $\Re \tilde m=m$ and the homogeneity for $|\xi|\geq 1$ means that $\alpha_{m-j}(x,\xi) = |\xi|^{\tilde m-j}\alpha_{m-j}(x,\xi/|\xi|)$ for any $(x,\xi)\in \bR^n\times\bR^n$ with $|\xi|\geq 1$. 
We write the poly-homogeneous expansion as $a\sim_h \sum_j a_{\tilde m-j} $ where $a_{\tilde m-j}(x,\xi)=|\xi|^{\tilde m-j}\alpha_{m-j}(x,\xi/|\xi|)\in C^\infty (\Omega \times (\bR^n\backslash\{0\}))$ is homogeneous of degree $\tilde m-j$ in $\xi$.  
We may call $a_{\tilde m}$ the (homogeneous) principal symbol of $a$ or of $A$ and $\tilde m$ the complex order of $a$ or $A$.

If the open set $\Omega$ is bounded,
we say that a symbol in $S^m(\Omega)$ is \emph{classical} with complex order $\tilde m$ when it admits a poly-homogeneous expansion with complex order $\tilde m$.
We denote by $S^{\tilde m}_{cl}$ the space of classical symbols with complex order $\tilde m$ and by $\Psi^{\tilde m}_{cl}=\Op(S^m_{cl})$ the space of classical pseudo-differential  operators with complex order $\tilde m$.

\medskip

If $F:\Omega_1\to \Omega_2$ is (smooth) diffeomorphism between two bounded open sets $\Omega_1,\Omega_2\subset \bR^n$, 
we keep the same notation for the map $F :C_c^\infty(\Omega_1)\to C_c^\infty(\Omega_1)$ given by $F (f) =f\circ F^{-1}$. 
For any $A\in \Psi^m(\Omega_1)$, the operator
$$
F^*A := F A F^{-1}  
$$
is then in $\Psi^m(\Omega_2)$.
This property allows us to define pseudo-differential operators on manifolds in the following way.
Let $M$ be a smooth compact connected manifold of dimension $n$ without boundary.
The space $\Psi^{m}(M)$ of pseudo-differential operators 
of order $m$ on $M$ is the space of operators which are locally transformed by some (and then any) coordinate cover to pseudo-differential operators in $\Psi^{m}(\bR^{n})$;
that is, the operator $A:\cD(M)\to \cD'(M)$ such that there exists a finite open cover $(\Omega_{j})_{j}$ of $M$,
a subordinate partition of unity $(\chi_{j})_{j}$
and diffeomorpshims $F_{j}:\Omega_{j}\to \cO_{j}\subset \bR^{n}$
that transform the operators $\chi_{k }A\chi_{j}:\cD(\Omega_{j})\to \cD'(\Omega_{k})$ into operators in $\Psi^{m}(\bR^n)$.

If $F:\Omega_1\to \Omega_2$ is (smooth) diffeomorphism between two bounded open sets $\Omega_1,\Omega_2\subset \bR^n$
and if $A\in \Psi^{\tilde m}_{cl}(\Omega_1)$, then the operator
$F^*A$
is then in $\Psi^{\tilde m}_{cl}(\Omega_2)$.
This property allows us to define pseudo-differential operators on manifolds in the following way.
The space $\Psi^{\tilde m}_{cl}(M)$ of classical  pseudo-differential operators 
of order $\tilde m$ on $M$ is the space of operators which are locally transformed by some (and then any) coordinate cover to classical pseudo-differential operators.

\subsection{Trace expansions}
\label{subsec_tr_exp}

In this section, we recall trace expansions for pseudo-differential operators.

We start with the trace and kernel expansions due to Seeley, Grubb and Schrohe.
Recall that a complex sector is a subset of $\bC\backslash\{0\}$ of the form 
$\Gamma=\Gamma_I:=\{re^{i\theta} \ :\  r>0,\ \theta\in I\}$
where $I$ is a subset of $[0,2\pi]$;
it is closed (in $\bC\backslash\{0\}$) when $I$ is closed.	

\begin{theorem}{\cite[Theorem 2.7]{Grubb+Seeley}}
\label{thm_GS}
Let $M$ be a compact smooth manifold of dimension $n\geq 2$
or let $\Omega$ be a bounded open subset in $\bR^n$.
Let $\cL\in \Psi_{cl}$ be an invertible elliptic operator of order $m_0\in \bN$. 
We assume that there exists a complex sector $\Gamma$ such that  the homogeneous principal symbol of $\cL$ in local coordinates satisfies  
$$
\ell_{m_0}(x,\xi)\notin -\Gamma^{m_0} =\{-\mu^{m_0} : \mu \in \Gamma\}
\quad \mbox{when}\ |\xi|=1.
$$

Let $A\in \Psi_{cl}^m$ and let $k\in \bN$ such that $-k m_0 +m<-n$.
The kernel $K(x,y,\lambda)$ of 	$A(\cL-\lambda)^{-k}$ is continuous and satisfies on the diagonal
\begin{equation}
\label{eq1_thm_GS}	
K(x,x,\lambda)\sim 
\sum_{j=0}^\infty c_j(x) \lambda^{\frac{n+m-j}{m_0}-k}
+ 
\sum_{l=0}^\infty \left(c'_l(x)\log \lambda +c''_l(x)\right) \lambda^{-l-k},
\end{equation}
for $\lambda \in -\Gamma^{m_0}$,
$|\lambda|\to \infty$, uniformly in closed sub-sectors of $\Gamma$.
The coefficients $c_j(x)$ and $c'_l(x)$ are determined from the symbols 
$a\sim_h\sum_j a_{m-j}$ and $\ell\sim_h \sum_j \ell_{m_0-j}$ in local coordinates, while the coefficients $c''_l(x)$ are in general globally determined.

As a consequence, one has for the trace
\begin{equation}
\label{eq2_thm_GS}	
\tr \left(A(\cL-\lambda)^{-k}\right)
\sim 
\sum_{j=0}^\infty c_j \lambda^{\frac{n+m-j}{m_0}-k}
+ 
\sum_{l=0}^\infty \left(c'_l\log \lambda +c''_l\right) \lambda^{-l-k},
\end{equation}
where the coefficients are the integrals over $M$ of the traces of the coefficients defined in \eqref{eq1_thm_GS}.
\end{theorem}

Integrating the expansion \eqref{eq2_thm_GS} against  $\lambda^z$ or against $e^{-z\lambda}$ on well chosen $z$-contours yields the following expansions for operators $A$ of any order, see also
 \cite[Section 1]{schrohe}:

\begin{theorem}
\label{thm_schrohe}
Let $\cL\in \Psi_{cl}$ be as in Theorem \ref{thm_GS}.

For any $A\in \Psi_{cl}^m$, we have for $t\to 0$:
\begin{equation}
\label{eq_heat}	
\tr \left(Ae^{-t \cL}\right)
\sim 
\sum_{j=0}^\infty \tilde c_j t^{\frac{n+m-j}{m_0}}
+ 
\sum_{l=0}^\infty \left(\tilde c'_l\ln t  +\tilde c''_l\right) t^{l},
\end{equation}
and 
\begin{equation}
\label{eq_power}	
\Gamma(t)\tr \left(A\cL^{-t}\right)
\sim 
\sum_{j=0}^\infty \frac{\tilde c_j}{s+\frac{n+m-j}{m_0}} t^{\frac{n+m-j}{m_0}}
+ 
\sum_{l=0}^\infty \left(\frac{-\tilde c'_l}{(t+l)^2} +
\frac{\tilde c''_l}{t+l}\right) .
\end{equation}
In \eqref{eq_power}, the left had side is meromorphic with poles as indicated by the right hand side.
The coefficients $\tilde c_j$, $\tilde c'_l$ and $\tilde c''_l$ 
are multiples of the corresponding $c_j$, $c'_l$ and $ c''_l$
in  \eqref{eq2_thm_GS}, 
the actors are universal constants independent of $A$ and $\cL$.
\end{theorem}

Recently, the author of this paper showed the following real trace expansions:
\begin{theorem}[\cite{myexp}]
\label{thm_myexp}
Let $\cL \in \Psi_{cl}^{m_0}$ be an elliptic self-adjoint operator  
on  a compact manifold $M$ of dimension $n\geq 2$.
Its order is $m_0>0$.
Let $A\in \Psi^m_{cl}$ and $\eta\in C_c^\infty(\bR)$.
Then the operator $A\ \eta(t \cL)$ is traceclass for all $t\in \bR$.
\begin{enumerate}
\item 
If $\eta$ is supported in $(0,\infty)$, then the trace  of $A\ \eta(t \cL)$ admits the following expansion as $t\to 0^+$, 
$$
\tr \left(A \eta(t\cL)\right)
\sim 
c_{m+n}
t^{-\frac{m+n}{m_0}} 
+
c_{m-n-1}
t^{-\frac{m+n-1}{m_0}} 
+
\ldots 
$$	
in the sense that 
$$
\tr \left(A \eta(t\cL)\right)
-\sum_{j=0}^{N-1}
c_{m+n-j}
t^{-\frac{m+n-j}{m_0}} 
=O(t^{-\frac{m+n-N}{m_0}} ).
$$	

\item 	
If  $\eta\equiv 1$ on a neighbourhood of $(\spec (\cL))\cap (-\infty,0]$
and if $\Re m\geq -n$ with $m \notin \bZ$,
then we have as $t\to 0^+$
 $$
 \tr (A\eta (t\cL)) = c'(A)+ \sum_{j=0}^{N-1} c'_{m+n-j} t^{\frac{-m-n+j}{m_0}}  +o(1),
 $$
where  $N\in \bN$ is  such that $\Re m+n<N$.
\item
In Part (1) , the constant $c_{m+n-j}$  are of the form 
$$
c_{m+n-j} = c_{m+n-j}^{(\sigma)} c_{m+n-j} ^{(\eta)}
$$
where $c_{m+n-j}^{(A)}$ depends only on the poly-homogeneous expansion of the symbol of $A$ in local coordinates and 
$$
\tilde c_{m+n-j}^{(\eta)}:=
\frac1{m_0}
\int_{u=0}^{+\infty}
\eta(u) \ u^{\frac{m-j+n}{m_0}} \frac{du}u ,
$$	
and similarly for Part (2).
\end{enumerate}
\end{theorem}

It is not known whether the constants in Theorems \ref{thm_myexp} and \ref{thm_schrohe} are related, 
except when they  are given 
by the non-commutative residue and the canonical trace,
see Propositions \ref{prop_res_in_exp} and \ref{prop_TR_in_exp}
respectively.

\subsection{The non-commutative residue}
\label{subsec_Wod}

Here, we recall the definition of the non-commutative residue via local symbols.
The original references are \cite{guillemin} and \cite{Wod_84,Wod_87}.
See also \cite{Grubb+Schrohe,schrohe,scott,Fed++Schrohe}.

Let $\Omega$ be a bounded open subset in $\bR^n$ with $n\geq 2$.
Let $A\in \Psi^m_{cl} (\Omega)$ with symbols $a\sim_h \sum_{j\in \bN_0} a_{m_j}$.
If $m\in \bZ_n$, we set
$$
\res_x(A) := \int_{\bS^{n-1}} a_{-n}(x,\xi) \ d\sig(\xi);
$$ 
in this paper, $\bZ_n$ denotes the set 
$$
\bZ_n:=\{-n,-n+1,-n+2,\ldots\},
$$
and 
$\sig$ denotes the surface measure on the Euclidean unit sphere $\bS^{n-1} \subset \bR^n$ which may be obtained as the restriction to $\bS^{n-1}$ of the $(n-1)$-form $\sig$ defined on $\bR^n$ by
$\sum_{j=1}^n (-1)^{j+1} \xi_j 
\ d\xi_1 \wedge \ldots \wedge d\xi_{j-1} \wedge d\xi_{j+1} \wedge \ldots \wedge d\xi_n$.
If $m\in \bC\backslash \bZ_n$, then we set $\res_x(A):=0$.

If $F:\Omega_1\to \Omega_2$ is (smooth) diffeomorphism between two bounded open sets $\Omega_1,\Omega_2\subset \bR^n$
and if $A\in \Psi^{m}_{cl}(\Omega_1)$, 
then
$$
|F'(x)|\res_{F(x)} (F^* A) =  \res_{x}(A).
$$ 
Hence $\res_x$ is a 1-density and makes sense on a compact manifold $M$.

\begin{definition}
\label{def_residue_density}
The function $x\mapsto \res_x A$ is the \emph{residue density} 
on a bounded open subset $\Omega\subset\bR^n$ or on an $n$-dimensional compact manifold $M$ with $n\geq 2$.
The corresponding integral 
$$
\res(A):=\int_M \res_x(A)
\quad\mbox{or}\quad
\res(A):=\int_\Omega \res_x(A) \ dx, 
$$
is called  the \emph{non-commutative residue} of $A$.
\end{definition}

The non-commutative residue is a trace on $\cup_{m\in \bC} \Psi^m_{cl}$
in the sense that it is a linear functional on $\cup_{m\in \bC} \Psi^m_{cl}$
which vanishes on commutators.
If $M$ is connected, then any other trace on $\cup_{m\in \bC} \Psi^m_{cl}$ is a multiple of $\res$.

The non-commutative residue also appears in the constant coefficients of the trace expansions recalled in Section \ref{subsec_tr_exp}:

\begin{proposition}
\label{prop_res_in_exp}
\begin{enumerate}
\item We continue with the setting and results of Theorem \ref{thm_GS}.
The coefficients $c'_0(x)$ and $c'_0$ in \eqref{eq1_thm_GS} and \eqref{eq2_thm_GS} satisfy
$$
c'_0 (x) =  \frac {(-1)^k}{(2\pi)^{n}m_0}\res_x (A) 
\qquad \mbox{and}\qquad
c'_0 =  \frac {(-1)^k}{(2\pi)^{n}m_0}\res (A).
$$
\item 
With the setting and results of Theorem
\ref{thm_schrohe}, we have
$$
\tilde c'_0 =  \frac {-\res (A)}{(2\pi)^{n}m_0}.
$$
\item 
With the setting and results of Theorem
\ref{thm_myexp}, we have when $m \in \bZ_n$
$$
c_0=\frac 1 {m_0} \res(A) \int_0^{+\infty} \!\!\! \eta(u)\ \frac{du}u .
$$  
\end{enumerate}
\end{proposition}

Parts (1) and (2) can be found in \cite[Section 1]{schrohe}, 
while Part (3) is part of the results in \cite{myexp}.

\subsection{The canonical trace}
\label{subsec_def_TR}

In this section,
 we recall the definition of the canonical trace.
References include the original paper  \cite{KV} by Kontsevich and Vishik, as well as \cite{Grubb+Schrohe,paycha}.

 Let $\Omega$ be an open bounded subset of $\bR^n$
and let $A=\Op(a)\in \Psi_{cl}^m(\Omega)$ with complex order $m\notin \bZ_n$.
The symbol of $A$ admits the poly-homogeneous expansion $a\sim_h \sum_{j\in \bN_0} a_{m-j}$.
For $x$ fixed, the function $a_{m-j}(x,\cdot)$ is  smooth  on $\bR^n\backslash\{0\}$ and  $(m-j)$-homogeneous with $m-j\notin \bZ_n$,
so \cite[Theorem 3.2.3]{hormander1} it extends uniquely into a tempered $(m-j)$-homogeneous distributions on $\bR^n$  for which we keep the same notation.
For each $x\in \Omega$, we define the tempered distributions
using the inverse Fourier transform 
$$
\kappa_{a,x}=\cF^{-1} \left\{a(x,\cdot)\right\}
\quad\mbox{and}\quad
\kappa_{a_{m-j},x}=\cF^{-1} \left\{a_{m-j}(x,\cdot)\right\}, 
\ j=0,1,2,\ldots
$$
The distribution $\kappa_{a_{m-j},x}$ is $(-n-m+j)$-homogeneous.
Then for any positive integer $N$ with $m-N<-n$ and $x\in \Omega$
the distribution $\kappa_{a,x} - \sum_{j=0}^{N}	\kappa_{a_{m-j},x}$
is a continuous function on $\bR^n$.
Furthermore, the function $(x,y)\mapsto \kappa_{a,x}(y) - \sum_{j=0}^{N}	\kappa_{a_{m-j},x}(y)$
is continuous and bounded on $\Omega\times \bR^n$.
Its restriction to $y=0$ is independent of $N>m+n$ and defines the quantity
$$
\TR_x(A) := \kappa_{a,x}(0) - \sum_{j=0}^{N}	\kappa_{a_{m-j},x}(0).
$$

If $F:\Omega_1\to \Omega_2$ is diffeomorphism between two bounded open sets $\Omega_1,\Omega_2\subset \bR^n$
and if $A\in \Psi^{\tilde m}_{cl}(\Omega_1)$, 
then
$$
|F'(x)|\TR_{F(x)} (F^* A) =  \TR_{x}(A).
$$ 
Hence, $\TR_x$ is a 1-density and makes sense on a compact manifold $M$.

\begin{definition}[\cite{KV}]
The density 	$x\mapsto\TR_x(A)$ on a compact manifold $M$ or a bounded open subset $\Omega$ is called the \emph{canonical trace density} of $A$.
The corresponding integral 
$$
\TR(A)=\int_M \TR_x(A)
\qquad\mbox{or}\qquad
\TR(A)=\int_\Omega \TR_x(A) dx,
$$
is called the \emph{canonical trace} of $A$.
\end{definition}

On $\Omega$ or $M$, the map $A\mapsto \TR(A)$ is a linear functional  on 
$\Psi^m_{cl}$ for each $m\in \bC\backslash \bZ_n$
and it coincides with the usual $L^2$-trace if $\Re m<-n$.
It is a trace type functional on 
$\cup_{m\in \backslash \bZ_n}\Psi^m_{cl}$ 
in the sense that 
$$
\TR(cA+dB) = c\TR(A) +d\TR(B) 
\qquad \mbox{whenever}\quad c,d\in \bC, \
 A,B\in 
 \cup_{m\in \backslash \bZ_n}\Psi^m_{cl},
  $$
 and 
 $$
 \TR(AB)=\TR(BA) 	
	 \qquad \mbox{whenever}\quad
	 AB,\ BA \in \cup_{m\in \backslash \bZ_n}\Psi^m_{cl}.
 $$

The canonical trace was originally defined  in \cite{KV}, and may be defined on  a slightly larger domain \cite{grubb}.
It is related with coefficients in the trace expansions recalled in Section \ref{subsec_tr_exp}:

\begin{proposition}
\label{prop_TR_in_exp}
\begin{enumerate}
\item We continue with the setting and results of Theorem \ref{thm_schrohe}. 
The coefficients  $c''_0$ in  \eqref{eq2_thm_GS} is equal to the canonical trace $\TR(A)$ of $A$ when $m\notin\bZ_n$.
\item 
With the setting and results of Theorem
\ref{thm_myexp},  when $m\notin\bZ_n$, 
we have $c'(A)=\TR(A)$.
\end{enumerate}
\end{proposition}

Part (1) can be found in \cite[Section 1]{Grubb+Schrohe} 
while Part (2) is part of the results in \cite{myexp}.

 The residue of the canonical trace of a (suitable) holomorphic family of classical pseudo-differential operators is equal to the non-commutative residue.
It can also be read off the zeta function of $A$ 
and in the $\Omega$-setting is related to the finite-part integral of the symbol of $A$ on $\Omega\times \bR^n$, see also \cite{lesch} and \cite{myexp}.

\section{The case of the torus}
\label{sec_T}

In this section, 
we discuss the relations between the pseudo-differential  calculi
defined on the torus viewed as a compact manifold  and defined via the Fourier series.

In this paper,  the $n$-dimensional torus is denoted by  $\bT^n$
and  is realised as $\bT^n=\bR^n/\bZ^n$.

\subsection{The toroidal pseudo-differential  calculus}
\label{subsec_psiT}

In this section we set the notation for and define the toroidal pseudo-differential  calculus.
We refer to \cite{ruzhansky+turunen_bk} for an in-depth presentation.

A toroidal symbol is a scalar function $\sigma$ defined on $\bT^n\times \bZ^n$.
The operator $\Op_{\bT^n} (\sigma)$ 
 associated with the toroidal symbol $\sigma$ is the operator defined via
 $$
 \Op_{\bT^n} (\sigma)f(x) 
 = \sum_{\ell\in \bZ^n}  e^{2i\pi x\cdot\ell} \sigma(x,\ell)\widehat f(\ell), 
 \qquad x\in \bT^n.
 $$
In this formula, $f$ is in  the space $L^2_{finite}(\bT^n)$ of smooth functions on $\bT^n$ whose Fourier coefficients 
$$
\widehat f(\ell) =\cF_{\bT^n} f(\ell)=\int_{\bT^n} f(x) e^{-2i\pi x\cdot\ell}dx,
\quad \ell\in \bZ^n,
$$
all vanish except for a finite number of them.
We will keep the same notation for the operator $\Op_{\bT^n}(\sigma)$ 
and its natural extensions between topological vector spaces containing $L^2_{finite}(\bT^n)$ as a dense subspace.

For each $j=1,\ldots, n$, we denote by $\Delta_j$ the difference operator in the $j$th direction, that is, the operator acting on toroidal symbols $\sigma$ in the following way:
$$
\Delta_j \sigma(x,\ell) = \sigma(x,\ell +e_j)-\sigma(x,\ell), 
\qquad (x,\ell)\in \bT^n\times \bZ^n,
$$
where $(e_1,\ldots,e_n)$ is the canonical basis of $\bR^n$.
The difference operator for the multi-index $\alpha=(\alpha_1,\ldots,\alpha_n)\in \bN_0^n$ is denoted by
$$
\Delta^\alpha := \Delta_1^{\alpha_1} \ldots \Delta_n^{\alpha_n},
$$
with the convention that $\Delta_j^0=\id$.

The following statements say that the pseudo-differential  operators defined locally on the manifold $M=\bT^n$  
have a global description as toroidal operators of the form $\Op_{\bT^n} (\sigma)$: 
\begin{theorem}[\cite{ruzhansky+turunen_bk}]
\label{thm:PsiT}
Let $A\in \Psi^m(\bT^n)$ for some $m\in \bR$.
Then there exists a unique toroidal symbol $\sigma_A$ such that $A=\Op_{\bT^n}(\sigma_A)$.
It satisfies:
\begin{equation}
\label{eq:SmT}	
\forall\alpha,\beta\in \bN_0^n\qquad
\exists C>0 \qquad
\forall (x,\ell)\in \bT^n\times \bZ^n\qquad
|\partial_x^\beta\Delta^\alpha \sigma_A(x,\ell)|
\leq C \langle \ell\rangle^{m-|\alpha|}.
\end{equation}
Conversely, if a toroidal symbol $\sigma$ satisfies \eqref{eq:SmT}, then $\Op_{\bT^n}(\sigma)\in \Psi^m(\bT^n)$. 
\end{theorem}

We denote by $S^m(\bT^n)$ the Fr\'echet space of toroidal symbols satisfying 
\eqref{eq:SmT} and we say then that the symbols are of order $m$.
There should be no confusion with the notation $S^m(\Omega)$ since there $\Omega$ is an open subset of $\bR^n$.
 Theorem \ref{thm:PsiT}  may then be rephrased as 
$$
\Psi^m(\bT^n)=\Op_{\bT^n}(S^m(\bT^n)).
$$
Furthermore, the proof of Theorem \ref{thm:PsiT} shows that 
the map $\sigma \mapsto \Op_{\bT^n}(\sigma)$ 
is an isomorphism from $S^m(\bT^n)$ to $\Psi^m(\bT^n)$. 

We denote by $S^{-\infty}(\bT^n)=\cap_{m\in \bR} S^m(\bT^n)$ the set of smoothing toroidal symbols.
As a consequence of the results mentioned above, the map  $\sigma \mapsto \Op_{\bT^n}(\sigma)$
is an isomorphism of Fr\'echet spaces from $S^{-\infty}(\bT^n)$ to $\Psi^{-\infty}(\bT^n)$.

Naturally, 
another way of defining globally an operator $A\in \cup_{m\in \bR} \Psi^m(\bT^n)$
is via its integral kernel $K_A \in \cD'(\bT^n\times\bT^n)$
or equivalently via the distribution given by
$k_{A,x}(y)=K_A(x,x - z)$.
Indeed, we have (in the sense of distributions)
for any $f\in C^\infty(\bT^n)$ and $x\in \bT^n$:
$$
Af(x)
=\int_{\bT^n} K_A(x,y) f(y) dy
=\int_{\bT^n}  f(y) \kappa_{A,x}(x-y) dy
=f*\kappa_{A,x}(x).
$$
Recall that the map
$x\mapsto \kappa_{A,x}\in \cD'(\bT^n)$ is smooth  on $\bT^n$
and that we have 
$$
\forall (x,\ell)\in \bT^n\times \bZ^n\qquad
\widehat \kappa_{A,x}(\ell)=\sigma_A (x,\ell)
\qquad\mbox{where} \ 
\sigma_A:=\Op_{\bT^n}^{-1}(A).
$$
Moreover, $\kappa_{A,x}$ is smooth away from the origin for $x$ fixed
since $K_A$ is smooth away from the diagonal $x=y$.
In fact, an operator $A:\cD(\bT^n)\to \cD'(\bT^n)$ is in $\Psi^{-\infty}(\bT^n)$
if and only if $(x,y)\mapsto \kappa_x(y)$ is smooth on $\bT^n\times\bT^n$.

If $\sigma_A$ does not depend on $x$, then $A$ is a Fourier multiplier with symbol $\sigma$ and convolution kernel $\kappa_A$. Even when $\sigma_A$ depends on $x$, we may abuse the vocabulary and call $\kappa_{A,x}$ the convolution kernel of $A$.

More generally, we have the following property between pseudo-differential and translations:
\begin{equation}
\label{eq:sigmatranslT}
\tau_{x_0}^{-1}\Op_{\bT^n}(\sigma)  \tau_{x_0} = \Op_{\bT^n}(\tau_{x_0}\sigma) 
\end{equation}
where $\tau_{x_0}$ denotes the translation
\begin{equation}
\label{eq_def_taux_0}	
\tau_{x_0} f(x) = f(x-x_0), \qquad x_0\in \bT^n, \ f\in L^2(\bT^n)
\end{equation}
and we abused the notation to have $\tau_{x_0}\sigma:(x,\ell)\mapsto \sigma(x-x_0,\ell)$.

As on $\bR^n$, 
we say that the toroidal symbol $\sigma\in S^m(\bT^n)$ admits an \emph{expansion} and we write $\sigma\sim \sum_{j\in \bN_0} \sigma_{m-j}$ when $\sigma_{m-j}\in S^{m-j}(\bT^n)$ and 
$\sigma-\sum_{j=0}^{N} \sigma_{m_j} \in S^{m-N-1}(\bT^n)$.

\subsection{Main results for the classical toroidal pseudo-differential  calculus}
\label{subsec_psiclT}

As in the Euclidean setting, 
we say that a toroidal symbol $\sigma$ in $S^m(\bT^n)$ is \emph{classical with complex order} $\tilde m\in \bC$ when it admits a poly-homogeneous expansion
 $\sigma\sim_h \sum_{j\in \bN_0} \sigma_{\tilde m-j}$, 
 that is,
when each toroidal symbol  $\sigma_{\tilde m-j}(x,\ell)$ is  $(\tilde m-j)$-homogeneous in $\ell\not=0$, i.e. $\sigma_{\tilde m-j}(x,r\ell) = r^{\tilde m-j}\sigma_{\tilde m -j}(x,\ell)$ for any $(x,\ell)\in \bT^n\times(\bZ^n\backslash\{0\})$, $r\in \bN$.
We denote by $S^{\tilde m}_{cl}(\bT^n)$ the space of classical symbols of order $\tilde m\in \bC$.

The next statement  says that the pseudodifferential operators corresponding to classical toroidal (global) symbols are exactly the classical ones on the manifold $M=\bT^n$:

\begin{theorem}
\label{thm:PsiclT}
For any $m\in \bC$, we have
$$\Psi^m_{cl}(\bT^n)=\Op_{\bT^n}(S^m_{cl}(\bT^n)).$$
\end{theorem}

Below, we will state and then prove a  stronger (but more technical) result for each classical pseudo-differential operators on $\bT^n$ which will imply Theorem \ref{thm:PsiclT}.
In order to state this stronger result, we need the following property which will allow us to identify a homogeneous toroidal symbol with a homogeneous symbol on $\bR^n$:

\begin{proposition}
	\label{prop_ext_sigma_hom}
If $\sigma$ is a toroidal symbol which is $m$-homogeneous in $\ell$ with $m\in \bC$, 
then there exists a unique continuous function 	in $(x,\xi)\in\bT^n\times(\bR^{n}\backslash \{0\})$
 which is $m$-homogeneous in $\xi$ and coincides with $\sigma$ 
on $ \bT^n\times(\bZ^n\backslash\{0\})$.
Still denoting by $\sigma$ this unique extension,  $\sigma $ is in fact smooth on $\bT^n\times (\bR^{n}\backslash \{0\})$.
\end{proposition}

Proposition \ref{prop_ext_sigma_hom} will be proved in Section \ref{subsec_pf_thmPsiclT}.
Note that as a consequence of Proposition \ref{prop_ext_sigma_hom}, 
a toroidal symbol which is $m$-homogeneous in $\ell$
is in $S^m_{cl}(\bT^n)$, $m\in \bC$.

The stronger result we alluded to above is the following.
Roughly speaking, it says that the poly-homogeneous expansion of a toroidal symbol is given by the poly-homogeneous expansion in local charts and vice versa:

\begin{proposition}
\label{prop_thm:PsiclT}
We fix a finite open covering $(\Omega_k)_k$ of $\bT^n$
and a subordinate partition of unity $(\chi_{k})_{k}$ of the form 
$\chi_k (x)= \psi_k(x-x_k)$ where the points 
$x_k\in \Omega_k$ are distinct and the function
$\psi_k\in \cD(\bT^n)$ is supported on a small neighbourhood of 0;
moreover, $0\leq \psi_k\leq 1$ and 
$\psi_k\equiv 1$ on an even smaller neighbourhood of 0.

Let $A\in \Psi^{m_1}(\bT^n)$. 
Then $\chi_{k_1} A\chi_{k_2}$ is smoothing for $k_1\not=k_2$
and if $k_1=k_2=k$, using the $x_k$ translations (see \eqref{eq_def_taux_0}) we have 
$$
\chi_k A\chi_k := \tau_{x_k}^{-1} A_k \tau_{x_k},
\qquad\mbox{where}\qquad 
A_k:=\psi_k \tau_{x_k}^{-1}  A \tau_{x_k}  \psi_k.
$$

Denoting by $\sigma=\Op_{\bT^n}^{-1} A$ the toroidal symbol of $A$, we have
$\sigma\in S^{m_1}(\bT^n)$ by Theorem \ref{thm:PsiT}.
 
\begin{enumerate}
\item 
If $\sigma$ is classical, i.e. $\sigma\in S^m_{cl}(\bT^n)$ with 
$\sigma\sim_h \sum_{j\in \bN_0} \sigma_{m-j}$, 
then the symbol of $A_k$ admits the poly-homogeneous expansions
$\sum_{j\in \bN_0} \psi_k(x) \sigma_{m-j}(x-x_k,\xi)$
where the functions $\sigma_{m-j}$ have been smoothly extended (see Proposition \ref{prop_ext_sigma_hom}).

\item 
If $A$ is classical on $\bT^n$ with complex order $m\in \bC$ then 
$A_k$ is classical on $\bR^n$, 
its symbol is supported in a small neighbourhood of $0$ in $x$ and
admits a poly-homogeneous expansion $\sum_{j\in \bN_0} a_{m-j}^{(k)}$;
moreover, $\sigma$ admits a poly-homogeneous expansion which coincides with 
$\sum_{j\in \bN_0} a_{m-j}^{(k)}(x-x_k,\ell)$ on $\{\chi_k\equiv1\}$
\end{enumerate}
\end{proposition}
In Proposition \ref{prop_thm:PsiclT} and in its proof
(given in Section \ref{subsec_pf_thmPsiclT}), we will allow ourselves to make no distinction between functions defined on $\bT^n$ or $\bR^n$ when they are supported in a small enough neighbourhood of $0$.

Note that Proposition \ref{prop_thm:PsiclT} and the local definition of the non-commutative residue (see Section \ref{subsec_Wod}) readily imply:

\begin{corollary}
\label{cor_res_thm:PsiclT}	
If $\sigma\in S^m_{cl}(\bT^n)$ with $m\in \bZ_n$, then 
$$
\res_x (A) = \int_{\bS^{n-1}} \sigma_{-n}(x,\xi) \ d\sig (\xi)
\quad\mbox{and}\quad
\res (A) = \int_{\bT^{n}\times \bS^{n-1}} \sigma_{-n}(x,\xi) \ dx d\sig (\xi).
$$
where $\sigma_{-n}$ is the smooth extension of the symbol  $\sigma_{-n}$ contained in the expansion
$\sigma\sim_h \sum_{j\in \bN_0} \sigma_{m-j}$ as granted by Proposition \ref{prop_ext_sigma_hom}.
\end{corollary}

The expansion results applied to the Laplacian $\Delta$ of $\bT^n$ yields easily:

\begin{corollary}
\label{cor_TR_thm:PsiclT}	
Let  $\sigma\in S^m_{cl}(\bT^n)$ with $m\in \bC$. 
We set $A:=\Op_{\bT^n}(\sigma)$.
\begin{enumerate}
\item Let $\eta\in \cD(0,\infty)$.
Then we have an expansion as $t\to 0^+$:
$$
\tr(A\, \eta(\Delta)) = \sum_{\ell\in \bZ^n} \sigma(x,\ell)\eta(t|\ell|^2)
\sim 
c_{m+n}
t^{-\frac{m+n}{m_0}} 
+
c_{m-n-1}
t^{-\frac{m+n-1}{m_0}} 
+
\ldots 
$$	
The constants $c_{m+n-j}$  are of the form 
$c_{m+n-j} = c_{m+n-j}^{(\sigma)} c_{m+n-j}^{(\eta)}$
where $c_{m+n-j}^{(\sigma)}$ depends only on the homogeneous expansion of $\sigma$ and 
$$
c_{m+n-j}^{(\eta)}:=
\frac12
\int_{u=0}^{+\infty}
\eta(u) \ u^{\frac{m-j+n}{2}} \frac{du}u .
$$	

With the same constants $c_{m+n-j}^{(\sigma)}$, if $\Re m> -n$ then
we have the expansion as $R\to+\infty$
$$
\sum_{R \leq |\ell| \leq 2 R  } \int_{\bT^n} \sigma(x,\ell) dx
=
 \sum_{j=0}^{N-1} \frac{2^{m-j+n} -1}{m-j+n}  
 c_{m+n-j}^{(\sigma)}  R^{-m+n-j} +o(1),
$$
where $N\in \bN$ is the smallest integer such that $\Re m+n-N<0$.

If  $m\in \bZ_n$ then $c_0=c_0^{(\sigma)}c_0^{(\eta)}$ with
$$
c_0^{(\sigma)}=\res A,\qquad\mbox{and}\qquad
c_0^{(\eta)}=\frac 12 \int_0^{+\infty} \!\!\! \eta(u)\ \frac{du}u ,
$$  
and 
we have the expansion as $R\to+\infty$
$$
\sum_{R \leq |\ell| \leq 2 R  } \int_{\bT^n} \sigma(x,\ell) dx
=
 \sum_{j=1}^{m+n} \frac{2^{m-j+n} -1}{m-j+n}  
 c_{m+n-j}^{(\sigma)}  R^{-m+n-j} 
 +
\ln 2 \ \res A 
 +o(1).
$$

\item 
Here, $\Re m\geq -n$ with $m \notin \bZ$
and $N\in \bN$ is  such that $\Re m+n<N$.
Let $\eta\in \cD(\bR)$ 	such that  $\eta\equiv 1$ on a neighbourhood of 
$0$.
Then we have an expansion as $t\to 0^+$:
$$
\tr(A\, \eta(\Delta)) = \sum_{\ell\in \bZ^n} \sigma(x,\ell)\eta(t|\ell|^2) = 
\TR(A) + \sum_{j=0}^{N-1} c'_{m+n-j} t^{\frac{-m-n+j}{m_0}}  +o(1).
$$
The constants $c'_{m+n-j}$  are of the form 
$c'_{m+n-j} = {c'}_{m+n-j}^{(\sigma)} {c'}_{m+n-j}^{(\eta)}$
where ${c'}_{m+n-j}^{(\sigma)}$ depends only on the poly-homogeneous expansion of $\sigma$ and 
$$
{c'}_{m+n-j}^{(\eta)}:=
\frac1{2}
\int_{u=0}^{+\infty}
\eta(u) \ u^{\frac{m-j+n}{2}} \frac{du}u .
$$	
With the same constants ${c'}_{m+n-j}^{(\sigma)}$, 
we have the expansion as $R\to+\infty$
$$
\sum_{|\ell|<R} \int_{\bT^n} \sigma(x,\ell) dx
=
\TR (A) + \sum_{j=0}^{N-1} \frac{{c'}_{m+n-j}^{(\sigma)}} {m-j+n} R^{-m+n-j} +o(1).
$$
Consequently,  the canonical trace of the operator $A=\Op_{\bT^n}(\sigma)$ is given by the discrete finite part of $\int_{\bT^n} \sigma(x,\ell) dx$.
\end{enumerate}

\end{corollary}

\begin{proof}[Proof of Corollary \ref{cor_TR_thm:PsiclT}]
The first results in Part (1) and (2) are directs applications of Theorem \ref{thm_myexp} and Proposition 
\ref{prop_TR_in_exp}.
For the expansion in $R=t^{-2}$, we apply these results to suitable functions $\chi_k\in C_c^\infty(\bR)$ 
approximating the indicator $1_I$  of the interval $I=[1,4]$ for Part (1) or of the interval $I=[0,1]$ for Part (2).
Suitable approximations are for instance functions satisfying 
 $0\leq \eta_k \leq 1$, $\eta_k\equiv 1$ on $I$ and $\eta_k\equiv 0$ outside $I + (-\frac 1k,\frac 1k)$.
As the Laplacian $\Delta$ is invariant under translation, we compute easily  for $R>1$
$$
\tr\left( A (1_{I} -\eta_k)(R^{-2}\Delta))\right)
=
\sum_{\ell \in \bZ^n}\int_{\bT^n} 
 \sigma(x,\ell)dx
\left( 1_{I}-\eta_k\right)(R^{-2}|\ell|^2).
$$
We also estimate for $I=[1,4]$
\begin{align*}
 \sum_{\ell \in \bZ^n}
 \langle\ell\rangle^{\Re m}
  \left( 1_{I}-\eta_k\right)(R^{-2}|\ell|^2) 
 &\leq 
 \sum_{\substack{R\sqrt {1-k} \leq |\ell|\leq R\\
 \mbox{\tiny or } 2R \leq |\ell|\leq 2R \sqrt {1+k}
 }}
  \frac 1{|B(\ell,1/2)|}\int_{B(\ell,1/2)}
  \langle\xi\rangle^{\Re m} d\xi
	\\
&\lesssim 
\int_{\substack{R\sqrt {1-k} -\frac 12 \leq |\xi|\leq R+\frac 12\\
 \mbox{\tiny or } 2R -\frac 12 \leq |\ell|\leq 2R \sqrt {1+k}+\frac 12
 }}
  \langle\xi\rangle^{\Re m} d\xi
  \lesssim R^{\Re m+n}\frac 1k,
\end{align*}
and similarly if $I=[0,1]$.
Hence, in both case, we have 
$$
\tr\left( A (1_{I} -\eta_k)(R^{-2}\Delta))\right) = R^{\Re m+n}O(\frac 1k).
$$
We check easily the convergences of the coefficients
${c'}_{m+n-j}^{(\eta_k)}\longrightarrow_{k\to+\infty}
{c'}_{m+n-j}^{(1_I)}$.
The results follow by taking  $k\to +\infty$. 
\end{proof}

A modification of the proofs of Theorem \ref{thm_myexp} Part (2) and Proposition 
\ref{prop_TR_in_exp} Part (2) given in \cite{myexp} adapted
to the case of the torus would also imply that the canonical trace density at $x$ of the operator $A=\Op_{\bT^n}(\sigma)$ is given by the discrete finite part of $ \sigma(x,\ell)$.
The results regarding the canonical trace on the torus have already been obtained in \cite[Section 5]{Paycha+_tams}.
The particular case of the limit of 
$\sum_{R \leq |\ell| \leq 2 R  } \int_{\bT^n} \sigma(x,\ell) dx$ in the case  $m=-n$ and $R=2^K$
have been shown in \cite[Theorem 11.15]{pietsch}.

\subsection{Proofs of Propositions \ref{prop_ext_sigma_hom} and
\ref{prop_thm:PsiclT}}
\label{subsec_pf_thmPsiclT}

The main ingredient in the proof of Proposition \ref{prop_ext_sigma_hom} is the following:
\begin{lemma}
\label{lem:sigma0}
Let $\sigma\in S^0(\bT^n)$ be a toroidal symbol which 
is independent of $x$ and 0-homogeneous in $\ell \in \bZ^n\backslash\{0\}$, 
that is, $\forall\ell\in \bZ^n\backslash\{0\}$ and $r\in \bN$, 
$\sigma(r\ell)=\sigma(\ell)$.
There exists a unique continuous function
 $\sigma \in C(\bR^{n}\backslash \{0\})$
which is 0-homogeneous and coincides with $\sigma$ 
on $ \bZ^n\backslash\{0\}$.
Furthermore $\sigma$ is smooth.
\end{lemma}

\begin{proof}[Proof of Lemma \ref{lem:sigma0}]
For any $\ell_0,\ell_1\in \bZ^n\backslash\{0\}$ such that
$[\ell_0]_j [\ell_1]_j \geq 0$ for $j=1,\ldots n$,  
the membership of $\sigma$ in $S^0(\bT^n)$ and comparison with integrals yield
$$
|\sigma(\ell_0)-\sigma(\ell_1)|\lesssim_{\sigma} |\ln \frac{|\ell_0|}{|\ell_1|}|.
$$
This together with the homogeneity of $\sigma$ implies  that $\sigma$ extends uniquely to a continuous function which is defined on $\bR^n\backslash\{0\}$
and $0$-homogeneous, and for which we keep the same notation $\sigma$. 

To show that $\sigma$ is differentiable, we start with the following observation: 
for every (fixed) $\xi_0\in \bQ^n\backslash\{0\}$, 
choosing $r_0\in \bN$ such that $\ell_0:=r_0\xi_0\in \bZ^n$, 
we have for any $p,q\in \bZ$, $0\leq |p|< q$:
\begin{equation}
\label{eq1_pf_lem:sigma0}
\left|\sigma\left(\xi_0 + \frac pq e_1\right) -\sigma(\xi_0) - r_0 p \Delta_{e_1} \sigma(q \ell_0)\right|
=
\left|\sum_{j=0}^{r_0p-1} \Delta_{je_1}\Delta_{e_1} \sigma(q\ell_0)\right|
\lesssim |q\ell_0|^{-2} (r_0|p|)^2,
	\end{equation}
having  
the membership of $\sigma$ in $S^0(\bT^n)$ and comparison with integrals.
The second observation is 
that for every $\ell\in \bZ^n\backslash\{0\}$, 
the sequence $(q\Delta_{e_1} \sigma(q\ell))_{q\in \bN}$ is bounded
because $\sigma\in S^0(\bT^n)$.
Therefore,we can extract a converging subsequence.
We can also assume that the subsequence $(q_j)_{j\in \bN}$ is the same for all $\ell\in \bZ^n\backslash\{0\}$; this yields a symbol in $S^{-1}(\bT^n)$.
As $|\ell|\in S^1(\bT^n)$, the symbol $\sigma_1$ defined via
$$
\sigma_1(\ell) := |\ell| \lim_{j\to \infty} q_j \Delta_{e_1} \sigma(q_j \ell), 
$$ 
is in $S^0(\bT^n)$.  
If $h\in (-1,1)$, 
applying \eqref{eq1_pf_lem:sigma0} with $p_j:=\lfloor hq_j\rfloor$ and $q_j$, then taking $j\to+\infty$, we obtain by continuity of $\sigma$
\begin{equation}
\label{eq2_pf_lem:sigma0}
\left|\sigma\left(\xi_0 + h e_1\right) -\sigma(\xi_0) -  \frac{h}{|\xi_0|} \sigma_1 (\ell_0)\right|
\lesssim |\xi_0|^2 h^2.
	\end{equation}
This shows that $\partial_1\sigma(\xi_0)$ exists  when $\xi_0\in \bQ^n\backslash\{0\}$.
Consequently, 
$(q\Delta_{e_1} \sigma(q\ell))_{q\in \bN}$ converges to $\partial_1 \sigma (\ell)$ as $q\to +\infty$ and this implies that the symbol given by $\partial_1 \sigma$ is  in $S^{-1}(\bT^n)$ and is $(-1)$-homogeneous. 
Therefore,the symbol $\sigma_1\in S^0(\bT^n)$ coincide with $\{|\ell| \partial_1 \sigma(\ell):\ell\in \bZ^n\}$ and is therefore 0-homogeneous. As in the first paragraph, we keep the same notation for its continuous extension to $\bR^n\backslash\{0\}$. 
The continuity of $\sigma$ and $\sigma_1$ together with \eqref{eq2_pf_lem:sigma0} implies that
$\partial_1\sigma$ exists and is continuous on $\bR^n\backslash\{0\}$.
Naturally the result is true for the other partial derivatives $\partial_2,\ldots,\partial_n$.
This implies that the function $\sigma$ is $C^1$ on $\bR^n\backslash\{0\}$.
We can also apply this result to $\sigma_1$ which is therefore also $C^1$.
Applying this result recursively and in all the directions shows that $\sigma$ is in fact smooth. 
\end{proof} 

The proof of Proposition \ref{prop_ext_sigma_hom} follows readily:

\begin{proof}[Proof of Proposition \ref{prop_ext_sigma_hom}]
	The case of $m=0$ is a simple adaptation of the proof of Lemma \ref{lem:sigma0} 
with the addition of  the smooth dependence in $x$; it is left to the reader.
For any $m\in \bC$, it suffices to apply the case $m=0$ to the symbol 
given by $|\ell|^{-m} \sigma(x,\ell)$ for $\ell\not=0$.
\end{proof}

We can now show Proposition \ref{prop_thm:PsiclT}.

\begin{proof}[Proof of Proposition \ref{prop_thm:PsiclT}]
We observe that the properties of translations (see \eqref{eq:sigmatranslT})
implies that it suffices to consider the case of $A_k$ with $k=0$.
The convolution kernel of $A_0$ is 
$$
\kappa^{(0)}_x(y)=\psi_0(x) \kappa_x(y) \psi_0(x-y).
$$
So the symbol of the pseudo-differential operator  $A_0\in \Psi^{m_1}(\bR^n)$ is $(x,\xi)\mapsto \cF_{\bR^n} \kappa_x^{(0)}(\xi)$ on $\bR^n$;
when restricted to $\xi\in \bZ^n$, this is the symbol of $A_0\in \Psi^{m_1}(\bT^n)$ because of the properties of the support of $\kappa^{(0)}$. Hence Part (2) follows. It remains to prove Part (1).
By refining the open covering, it suffices to prove the properties in Parts (1)  for a neighbourhood of 0 included in $\{\psi_0\equiv 1\}$.
Let $V$ be a neighbourhood of 0 included in $\{\psi_0\equiv 1\}$
and let  $\chi\in \cD(\bT^n)$ valued in $[0,1]$ with $\chi\equiv 1$ near 0 but supported in a neighbourhood of 0 so small that $V+ \supp \chi \subset \{\psi_0\equiv 1\}$. 
Then for $x\in V$, 
we have
$$
\kappa^{(0)}_x(y)= \kappa_x(y)\chi(y) 
+\kappa_x(y)(1-\chi)(y)  \psi_0(x-y).
$$
The second term is a smoothing kernel, so the symbol of $A_k$ is given modulo $S^{-\infty}$ by
\begin{align*}
\cF_{\bR^n}\kappa_x(y)\chi(y) 
&=
\int_{\bR^n} 
 \kappa_x(y) \chi(y)
 e^{-2i\pi y\xi} dy 
=
\int_{\bT^n}\ldots
=
\sum_{\ell\in \bZ^n}
\sigma(x,\ell) 
\overline{ \cF_{\bT^n} \{\chi(y) e^{2i\pi \cdot  \xi}\}(\ell)}
\\&=
\sum_{\ell\in \bZ^n}\sigma(x,\ell) 
\overline{\widehat \chi(\ell-\xi)},
\end{align*}
by Parseval formula for Fourier series,
where $\widehat \chi := \cF_{\bR^n}\chi \in \cS(\bR^n)$.
 
We assume that $\sigma$ is $m$-homogeneous in $\ell\not=0$
and we keep the same notation for its smooth homogeneous extension to $\bT^n\times (\bR^n\backslash\{0\})$ granted by Proposition \ref{prop_ext_sigma_hom}.
The Poisson summation formula and the properties of $\chi$  yield
$$
\sum_{\ell\in \bZ^n}\widehat\chi (\ell-\xi) 
=1
\quad\mbox{and}\quad
\sum_{\ell\in \bZ^n}(\xi -\ell)^\alpha\widehat\chi (\ell-\xi) 
=0,
$$
for any $\alpha\in \bN_0^n\backslash\{0\}$ and $\xi\in \bR^n$.
Therefore, we have for any $N\in \bN_0$
$$
\cF_{\bR^n} \{\chi \kappa_x\}(\xi)-
\sigma(x,\xi) 
=
\sum_{\ell\in \bZ^n}
 \left(\sigma(x,\ell) - \sum_{|\alpha|\leq N} \frac{(\xi -\ell)^\alpha}{\alpha!}  \partial_\xi^\alpha\sigma(x,\xi)  \right) 
 \overline{ \widehat\chi (\ell-\xi)}.
$$
For $r>0$, let $A_r:=\{\ell\in \bZ^n: r/2 \leq |\ell|\leq 2r\}$ denote the set of integers in the annulus with radii $r/2$ and $2r$.
We decompose the last sum above as 
$\sum_{\ell \notin A_{|\xi|}}+\sum_{\ell\in A_{|\xi|}}$
and we assume $|\xi|\geq 1$ large.
 As $\sigma \in C^\infty(\bT^n\times (\bR^n\backslash\{0\}))$ is $m$-homogeneous, 
 and $\widehat \chi\in \cS(\bR^n)$, 
  for any $N_1\in \bN$,
\begin{align*}
| \sum_{\ell \notin A_{|\xi|}}|
& \lesssim
\sum_{\ell \notin A_{|\xi|}} (|\ell|^{\Re m} + \langle\xi-\ell\rangle^N |\xi|^{\Re m})
 \langle \xi -\ell\rangle^{-N_1}
\\
& \lesssim
 \sum_{|\ell|\leq |\xi|/2}
 |\xi|^{\Re m+N-N_1}
+ \sum_{|\ell|\geq 2|\xi|}
 |\ell|^{\Re m+N-N_1}
 \lesssim |\xi|^{\Re m+N-N_1+n}.
\end{align*}
 We write the second sum $\sum_{A_{|\xi|}}$ 
 as $\sum_{A_{|\xi|}\cap B(\xi, |\xi|^{1/2})} +\sum_{A_{|\xi|}\cap \ {}^c\! B(\xi, |\xi|^{1/2})}  $.
  The Taylor estimates and the homogeneity of $\sigma$ imply
 $$
 |\sum_{A_{|\xi|}\cap B(\xi, |\xi|^{1/2})}|
 \lesssim 
 \sum_{A_{|\xi|}\cap B(\xi, |\xi|^{1/2})}
 |\xi-\ell|^{N+1} |\xi|^{\Re m-(N+1)}
 \lesssim 
 |\xi| ^{\Re m-\frac{N+1}2 +n}.
 $$
 For the same reasons and because $\widehat \chi\in \cS(\bR^n)$, we have 
 $$
 |\sum_{A_{|\xi|}\cap \ {}^c\!B(\xi, |\xi|^{1/2})}|
 \lesssim 
 \sum_{A_{|\xi|}\cap \ {}^c\!B(\xi, |\xi|^{1/2})}
 |\xi-\ell|^{N+1} |\xi|^{\Re m-(N+1)} \langle \xi -\ell\rangle^{N}
 \lesssim 
 |\xi| ^{\Re m-\frac{N}2 +n}.
 $$
 
The estimates above imply that for all $N'\in \bN$ and $\xi\in \bR^n$, $|\xi|\geq 1$, we have
$$
|\cF_{\bR^n} \{\chi \kappa_x\}(\xi)-
\sigma(x,\xi)|\lesssim_{\sigma,N'}  |\xi|^{-N'}.
$$
We can estimate in the same way 
$\partial_x^\beta\partial_\xi^\alpha \left\{\cF_{\bR^n} \{\chi \kappa_x\}(\xi)-
\sigma(x,\xi)\right\}$.
This shows that $\cF_{\bR^n} \{\chi \kappa_x\}(\xi)=\chi(x)\sigma$ modulo $S^{-\infty}$ for $x\in V$.
By linearity, it is true for any symbol $\sigma\in S^m_{cl}(\bT^n)$ with 
$\sigma\sim_h \sum_{j\in \bN_0} \sigma_{m-j}$.
This concludes the proof of Proposition \ref{prop_thm:PsiclT}.
\end{proof}

\section{Compact Lie groups}
\label{sec_G}

In this section we extend to the non-commutative setting some of the results that were proved in Section \ref{sec_T}.

\subsection{Notations and conventions}
\label{subsec_not_sec_G}
In Section \ref{sec_G}, $G$ always denotes a connected compact Lie group and $n>1$  its dimension.
In this section we set the notation and convention regarding the group $G$.
References include  \cite{hall_bk,helgason_bk,knapp_bk},
see also
\cite{monJFA,monarxiv}.

\subsubsection{Representations}
In this paper, a representation of $G$ is any continuous (hence smooth) group homomorphism $\pi$ from $G$ to the set of automorphisms of a finite dimensional complex vector space.
We will denote this space $\cH_\pi$ with dimension $d_\pi=\dim \cH_\pi$.
We may equip $\cH_\pi$ of an inner product $(\cdot, \cdot)_{\cH_\pi}$ for which $\pi$ is unitary.
The coefficients of $\pi$ are any function of the form 
$x\mapsto (\pi(x)u,v)_{\cH_\pi}$ for $u,v\in \cH_\pi$;
these are smooth functions on $G$.
If a basis $\{e_1,\ldots, e_{d_\pi}\}$ of $\cH_\pi$ is fixed, 
then we may identify $\cH_\pi$ with $\bC^{d_\pi}$, 
and the matrix coefficients of $\pi$
are the coefficients $\pi_{i,j}$, $1\leq i,j\leq d_\pi$ given by 
$\pi_{i,j}(x)=(\pi(x)e_i,e_j)_{\cH_\pi}$.

We will often identify a representation of $G$ and its class of equivalence.

We denote by 
$\RepG$ and $\Gh$
 the sets of representations of $G$
 and of irreducible representations of $G$
respectively, both 
 modulo equivalence.

\subsubsection{Lie algebra and vector fields}
The Lie algebra $\fg$ is the tangent space of $G$ at 
the neutral element $e_G$.
We denote by $\exp_G:\fg\to G$ the exponential mapping.

We may identify the Lie algebra $\fg$ 
with  the space of left-invariant vector fields on $G$.
If an basis $\{X_1,\ldots,X_n\}$ for $\fg$ is fixed,
we set for any multi-index $\alpha=(\alpha_1,\ldots,\alpha_n)\in \bN_0^n$
$$
X^\alpha:= X_1^{\alpha_1}\ldots X_n^{\alpha_n}.
$$

If $\pi$ is a representation of the group $G$, then 
$$
\pi(X)=\frac{\partial}{\partial t} \pi(\exp_G(t X))|_{t=0}, \qquad X\in \fg,
$$
defines a representation also denoted $\pi$ of $\fg$ and therefore of its universal enveloping Lie algebra.

The Lie algebra $\fg$ is the direct sum of the semi-simple Lie algebra $\fg_{ss}=[\fg,\fg]$ with its centre $\fz$:
$$
\fg=\fg_{ss}\oplus \fz.
$$
Note that $\fz$ is also the Lie algebra of the centre $Z_G$ of the group $G$.
The Killing form is a scalar inner product on 
 $\fg_{ss}\subset \fg$ and we extend it into a scalar inner product $(\cdot,\cdot)_{\fg}$ on $\fg$ which is invariant under the adjoint action of $G$.

The Casimir element of the universal enveloping algebra is 
$X_1^2+\ldots+X_n^2$ for one and then any orthonormal basis $X_1,\ldots,X_n$ of $\fg$.
Up to a sign, the corresponding differential operator is the (positive) Laplace-Beltrami operator of the compact Lie group $G$:
$$
\cL_G:= -X_1^2-\ldots- X_n^2.
$$
The Laplace-Beltrami operator $\cL_G$ is a non-negative essentially self-adjoint operator
on $L^2(G)$.
It is a central operator (i.e. invariant under left and right translations) and, for any irreducible representation $\pi$ of $G$, $\pi(\cL)$ is scalar:
$$
\pi(\cL_G(\pi)) =\lambda_\pi\id_{\cH_\pi}, \quad \pi\in \Gh.
$$

\subsubsection{Group Fourier transform}
If $f\in \cD'(G)$ is a distribution and $\pi$ is a unitary representation, 
we can always define its \emph{group Fourier transform} at $\pi$
denoted by
$$
\pi(f)\equiv \widehat f(\pi)\equiv \cF_G f(\pi) \in \sL(\cH_\pi), \quad \pi\in \Gh,
$$
 via
$$
 \pi(f)  = \int_G f(x)\pi(x)^* dx,
\quad \mbox{i.e.}\quad
(\pi(f)u,v)_{\cH_\pi} = \int_G f(x)(u,\pi(x)v)_{\cH_\pi} dx,
$$
since the coefficient functions are smooth.
In this paper, the Haar measure $dx$ on $G$ is normalised to be a probability measure.

We denote by $\L2f$ 
the set 
formed by the finite linear combinations of  coefficients of representations $\pi\in \Gh$.
Hence $\L2f$ is a vector subspace of  $\cD(G)$.
The Peter-Weyl Theorem states that $\L2f$ is dense in $L^2(G)$ and that 
for any $\pi,\pi'\in \Gh$, $u,v\in \cH_\pi$ we have
$$
\cF_G \{(\pi (\cdot) u,v)_{\cH_\pi}\}(\pi')
=
\left\{\begin{array}{ll}
d_\pi ^{-1} (u,v)_{\cH_\pi}&\mbox{if} \ \pi'=\pi,
\\
0	&\mbox{otherwise}
\end{array}\right.
$$
Consequently, a function $f\in \cD(G)$ admits the following expansion in Fourier series:
$$
f(x) = \sum_{\pi\in \Gh} d_\pi \tr \left(\pi(x) \widehat f(\pi)\right)
=\sum_{\pi\in \Gh} d_\pi
\sum_{1\leq i,j\leq d_\pi}
\pi_{i,j}(x) [\widehat f(\pi)]_{j,i}.
$$ 
Furthermore the Peter-Weyl Theorem
yields an explicit spectral decomposition for $\cL_G$.

\subsection{The H\"ormander pseudo-differential  calculus on $G$} 
\label{subsec_psiG}

In this section we discuss the H\"ormander pseudo-differential  calculus on $G$.
It may be viewed as a generalisation of the symbols and operator defined
on tori in Section \ref{sec_T}.

\subsubsection{The symbols and the natural quantisation}
The natural quantisation and notion of symbols on Lie groups
was first mentioned by Michael Taylor \cite{taylor_bk86}.

On a (connected) compact Lie group $G$, 
a symbol is a field 
$\sigma=\{\sigma(x,\pi) \in \sL(\cH_\pi), (x,\pi)\in G\times \Gh\}$
(for a general definition of fields of operators, see e.g. \cite[Part II Ch 2]{dixmierVnA}).
Its associated operator is the operator $\Op_G(\sigma)$ defined on $\L2f$ via
$$
\Op_G(\sigma) \phi (x)=\sum_{\pi\in \Gh}
d_\pi \tr \left(\pi(x) \sigma(x,\pi) \widehat \phi(\pi)\right),
\quad
\phi\in \L2f, \ x\in G.
$$

The Peter-Weyl theorem implies that for $\sigma(x,\pi)=\id_{\cH_\pi}$ the corresponding operator is the identity. Furthermore,
if $T$ is a linear operator defined on $\L2f$
(and with image some complex-valued functions of $x\in G$), 
then one recovers the symbol via
$$
\sigma(x,\pi)=
\pi (x)^* (T \pi)(x), 
\quad \mbox{that is,}\quad
[\sigma(x,\pi)]_{i,j}=
\sum_{k} \overline{\pi_{ki} (x)} (T \pi_{kj})(x) ,
$$
when one has fixed a matrix realisation of $\pi$.
This shows that the quantisation $\Op_G$ defined above is injective.

\subsubsection{Difference operators}
Here we recall the two notions  of difference operators, 
the RT-ones introduced by Michael Ruzhansky and Ville Turunen \cite{ruzhansky+turunen_bk}
and the fundamental ones introduced by the author in \cite{monJFA}.

If $q\in \cD(G)$, 
then the corresponding RT-difference operator
$\Delta_q$ is defined as acting on Fourier coefficients i.e. $\cF_G(\cD'(G))$ via
$$
\Delta_q \widehat f = \cF_G \{q f\},
\quad f\in \cD'(G).
$$ 
A collection 
$\Delta = \{\Delta_1,\ldots,\Delta_{n_{\Delta}}\}= \{\Delta_{q_1},\ldots,\Delta_{n_\Delta}\}$
 of RT-difference operators 
is strongly admissible when 
$$
\rank (\nabla_{e_G} q_1, \ldots, \nabla_{e_G} q_{n_\Delta})=n 
\quad\mbox{and}\quad
\{e_G\}=\cap_{j=1}^{n_\Delta} \{x\in G: q_j(x) =0\} .
$$
Such a collection exists, see \cite{ruzhansky+turunen+wirth} and \cite{monJFA}.
If $\alpha\in \bN_0^{n_\Delta}$, we write 
$\Delta^\alpha=\Delta_1^{\alpha_1}\ldots \Delta_{n_\Delta}^{\alpha_{n_\Delta}}$.
Note that the RT-difference operators are defined to act not on all the symbols but only on $\cF_G(\cD'(G))$. 

Let us now recall the definition of the fundamental difference operators as introduced by the authors in \cite{monJFA}, see also \cite{monarxiv}.
A symbol $\sigma=\{\sigma(x,\pi)\in \sL(\cH_\pi): (x,\pi)\in G\times \Gh\}$ 
is a field over the set $G\times \Gh$ and the Hilbert space $\oplus_{\pi\in \Gh}\cH_\pi$ extends naturally to 
a field over $G\times \RepG$ via
$$
\sigma(x,\pi_1\oplus \pi_2)
= 
\sigma(x,\pi_1)
\oplus
\sigma(x,\pi_2)
\in \sL(\cH_{\pi_1} \oplus \cH_{\pi_2}).
$$
Conversely, fields over $G\times \RepG$ satisfying such relations defines a symbol.
The difference operator with respect to ${\pi_0}\in \RepG$ is defined via
$$
\Delta_{\pi_0} \sigma (x,\pi) := \sigma(x,{\pi_0}\otimes \pi) - 
\sigma(x,\id_{\cH_{\pi_0}}\otimes \pi),
\qquad x\in G, \pi\in \RepG.
$$
This defines a field over the set $G\times \Gh$ and the Hilbert space $\oplus_{\pi\in \Gh}\cH_\pi\otimes \cH_{\pi_0}$.
We also extend this definition to any field over the set $G\times \Gh$ and the Hilbert space $\oplus_{\pi\in \Gh}\cH_\pi\otimes \cH_\varphi$ for any $\varphi\in \RepG$.
This allows us to compose difference operators.
We fix a finite set $\{\varphi_1,\ldots,\varphi_f\}\subset \Gh$ of representations of $G$
such that any representation in $\Gh$ will occur in a tensor products of representations in $\FundG$.
This exists \cite{monarxiv}, and we call these representations $\varphi_1,\ldots,\varphi_f$ and the corresponding difference operators $\Delta_{\varphi_1},\ldots,\Delta_{\varphi_f}$ fundamental.
If $\alpha\in \bN_0^f$, we write 
$\Delta^\alpha=\Delta_{\varphi_1}^{\alpha_1}\ldots\Delta_{\varphi_f}^{\alpha_f}$.

\subsubsection{Pseudo-differential  calculus on $G$}

The next statement characterises the global symbol of an operator in the H\"ormander pseudo-differential  calculus on $G$:
\begin{theorem}
\label{thm:PsiG}
\begin{enumerate}
\item Let $A\in \Psi^m(G)$ for some $m\in \bR$.
Then there exists a unique symbol $\sigma_A$ such that $A=\Op_G(\sigma_A)$.
	\begin{itemize}
\item 
For any basis $X_1,\ldots,X_n$ of $\fg$ and for any strongly admissible family of RT-difference operators $\Delta$, 
the symbol $\sigma_A$ satisfies
\begin{align}
\label{eq:SmG-RT}	
\forall\alpha\in \bN_0^{n_\Delta},\ \beta\in \bN_0^n\quad
\exists C>0 \quad
\forall (x,\pi)\in G\times \Gh\\
\|X^\beta\Delta^\alpha \sigma_A(x,\pi)\|_{\sL(\cH_\pi)}
\leq C \langle \lambda_\pi\rangle^{\frac{m-|\alpha|}2}.
\nonumber
\end{align}
\item 
The symbol $\sigma_A$ satisfies
\begin{align}
\label{eq:SmG}	
\forall \alpha \in \bN_0^f,\ \beta\in \bN_0^n\quad
\exists C>0 \qquad
\forall (x,\pi)\in G\times \Gh\\
\|X^\beta\Delta^\alpha \sigma_A(x,\pi)\|_{\sL}
\leq C \langle \lambda_\pi\rangle^{\frac{m-a}2}.\nonumber
\end{align}
Here $X_1,\ldots,X_n$ denotes a basis of $\fg$
and we have fixed fundamental representations $\varphi_1,\ldots,\varphi_f$.
\end{itemize}
\item Conversely, the following holds:
\begin{itemize}
\item 
If a symbol $\sigma$ satisfies \eqref{eq:SmG-RT}
for a basis $X_1,\ldots,X_n$ of $\fg$ and for a strongly admissible family of RT-difference operators $\Delta$, 
 then $\Op_{G}(\sigma)\in \Psi^m(G)$
 and it satisfies \eqref{eq:SmG-RT} for any basis $X_1,\ldots,X_n$ of $\fg$ and for any strongly admissible family of RT-difference operators $\Delta$.
\item 
If a symbol $\sigma$ satisfies \eqref{eq:SmG-RT} for a basis $X_1,\ldots,X_n$ of $\fg$ and for a set fundamental representations $\{\varphi_1,\ldots,\varphi_f\}$ of fundamental representations
 then $\Op_{G}(\sigma)\in \Psi^m(G)$
 and it satisfies \eqref{eq:SmG} for any basis of $\fg$ and any set of fundamental representations. 
\end{itemize}
\end{enumerate}
\end{theorem}

Theorem \ref{thm:PsiG} is shown in \cite{monJFA}.
Part (1) is also stated in 
\cite{ruzhansky+turunen_bk,ruzhansky+turunen+wirth},
unfortunately with proofs relying on properties of the pseudo-differential calculus (e.g. composition or properties equivalent to the statement above) which had not been shown in these works.

We denote by $S^m(G)$ the Fr\'echet space of symbols satisfying 
\eqref{eq:SmG} or equivalently \eqref{eq:SmG-RT}, and we say then that the symbols are of order $m$.
This extends the notation for the toroidal case $G=\bT^n$ viewed in Section \ref{subsec_psiT}.
As in the case of the torus, 
 Theorem \ref{thm:PsiG}  may then be rephrased as 
$\Psi^m(G)=\Op_{G}(S^m(G))$.
Furthermore, the proof of Theorem \ref{thm:PsiT} in \cite{monJFA} shows that 
the map $\sigma \mapsto \Op_{G}(\sigma)$ 
is an isomorphism of Fr\'echet spaces from $S^m(G)$ to $\Psi^m(G)$. 
This is so for any $m\in \bR$, and also for $m=-\infty$
having denoted by $S^{-\infty}(G)=\cap_{m\in \bR} S^m(G)$ the set of smoothing  symbols.

As on $\bR^n$ or $\bT^n$,
we say that the symbol $\sigma\in S^m(G)$ admits an expansion and we write $\sigma\sim \sum_{j\in \bN_0} \sigma_{m-j}$ when $\sigma_{m-j}\in S^{m-j}(G)$ and 
$\sigma-\sum_{j=0}^{m_N} \sigma_{m_j} \in S^{m-{N+1}}(G)$.

\subsubsection{Kernels}

As on the torus,
another way of defining globally an operator $A\in \cup_{m\in \bR} \Psi^m(G)$
is via its integral kernel $K_A \in \cD'(G\times G)$
or equivalently via the distribution given by
$k_{A,x}(y)=K_A(x,xz^{-1})$.
Indeed, we have (in the sense of distributions)
for any $f\in C^\infty(G)$ and $x\in G$:
$$
Af(x)
=\int_{G} K_A(x,y) f(y) dy
=\int_{G}  f(y) \kappa_{A,x}(y^{-1}x) dy
=f*\kappa_{A,x}(x).
$$
Recall that the map
$x\mapsto \kappa_{A,x}\in \cD'(G)$ is smooth  on $G$
and that we have 
$$
\forall (x,\pi)\in G\times \Gh\qquad
\widehat \kappa_{A,x}(\ell)=\sigma_A (x,\pi)
\qquad(\mbox{where} \ 
\sigma_A:=\Op_{G}^{-1}(A)).
$$
Moreover, $\kappa_{A,x}$ is smooth away from the origin for $x$ fixed
since $K_A$ is smooth away from the diagonal $x=y$.
In fact, an operator $A:\cD(G)\to \cD'(G)$ is in $\Psi^{-\infty}(G)$
if and only if $(x,y)\mapsto \kappa_x(y)$ is smooth on $G\times G$.

If $\sigma_A$ does not depend on $x$, then $A$ is a group Fourier multiplier with symbol $\sigma$ and convolution kernel $\kappa_A$. Even when $\sigma_A$ depends on $x$, we may abuse the vocabulary and call $\kappa_{A,x}$ the convolution kernel of $A$.

\begin{lemma}
\label{lem_tr_PsimG<-n} 
Let $A=\Op_G(\sigma_A)\in \Psi^{m}(G)$ with $m<-n$. Then 
$\sum_{\pi\in \Gh} d_\pi |\tr \int_G\sigma(x,\pi) dx|$ is finite and $A$ is trace-class with 
$$
\tr (A)=\int_G K_A(x,x)dx = \int_G \kappa_x(e_G) dx
=\sum_{\pi\in \Gh} d_\pi \tr \int_G\sigma(x,\pi) dx.
$$
\end{lemma}

\begin{proof}
The first equality holds on any compact manifold and any operator of order  $m<-n$.
As $m<-n$, $\kappa_x$ is continuous on $G$ so 
$\int_G K_A(x,x)dx=\int_G \kappa_x(e_G) dx$.

As $\sigma\in S^m(G)$, we have
$$
|\tr \int_G\sigma(x,\pi) dx|\leq \sup_{(x',\pi')\in G\times\Gh} \|\sigma(x',\pi')\|(1+\lambda_{\pi'})^{-\frac m2} \ 
\tr \left((1+\lambda_\pi)^{-\frac m2} \id_{\cH_\pi}\right).
$$
Recall that the convolution kernel $\cB_s$ of the operator $(\id+\cL)^{-\frac s2}$ is square integrable for $s>n/2$ \cite[Lemma A5]{monJFA} and that 
$$
\|\cB_s\|_{L^2(G)}^2 
= \sum_{\pi\in \Gh} d_\pi 
\| (1+\lambda_\pi)^{-\frac s2} \id_{\cH_\pi}\|_{HS}^2
= \sum_{\pi\in \Gh} d_\pi 
\tr (1+\lambda_\pi)^{-s} \id_{\cH_\pi}.
$$
Therefore, as $m<-n$, we have
\begin{align*}
\sum_{\pi\in \Gh} d_\pi |\tr \int_G\sigma(x,\pi) dx|
\leq \sup_{(x',\pi')\in G\times\Gh} \|\sigma(x',\pi')\|(1+\lambda_{\pi'})^{-\frac m2} 
\sum_{\pi\in \Gh} d_\pi \tr \left((1+\lambda_\pi)^{-\frac m2} \id_{\cH_\pi}\right)
\\=\sup_{(x',\pi')\in G\times\Gh} \|\sigma(x',\pi')\|(1+\lambda_{\pi'})^{-\frac m2}  \ \|\cB_{-m/2}\|^2,
\end{align*}
and the sum on the right-hand side is finite.
Moreover, 
since $\int_G\sigma(x,\pi)dx=\cF_G\{\int_G \kappa_x dx \}(\pi)$, the Fourier inversion formula yields:
$$
\int_G \kappa_x(e_G) dx=
\sum_{\pi\in \Gh} d_\pi \tr \int_G\sigma(x,\pi) dx.
$$
\end{proof}

\subsubsection{Invariance of the calculi under translations}
\label{subsec_invariance}

For any operator $A=\Op_G(\sigma_A)$ and $x_0\in G$,
we denote by ${}_{x_0}A$ and $A_{x_0}$ the left and right translated of $A$, that is, the operators given by
$$
{}_{x_0}A(f) (x)=A(f(x_0^{-1}\cdot) ) (x_0x)
\quad\mbox{and}\quad
A_{x_0} (f)(x) = A (f(\cdot \, x_0^{-1}) ) (x x_0).
$$
We check easily that 
the symbols of ${}_{x_0}A$ and $A_{x_0}$ are given by respectively:
$$
\sigma_{{}_{x_0}A} (x,\pi) = \sigma_A(x_0x,\pi)
\quad\mbox{and}\quad
\sigma_{A_{x_0}}(x,\pi) = \pi(x_0)\sigma_A(xx_0,\pi)\pi(x_0)^{-1}.
$$

It follows readily from the definition of the symbol classes (see the conditions in \eqref{eq:SmG}) that for any $x_0\in G$ we have
$$
\sigma_A \in S^m(G) \Longrightarrow 
\sigma_{{}_{x_0}A} \ \mbox{and}\ \sigma_{A_{x_0}} \ \mbox{are in} \ S^m(G).
$$
Furthermore, the maps 
$\sigma_A \mapsto \sigma_{{}_{x_0}A}$
and
$\sigma_A \mapsto \sigma_{A_{x_0}}$
are continuous isomorphisms of the Fr\'echet space $S^m(G)$.
Consequently, $\Psi^m(G)$ is invariant under left or right translations in the sense that for any $x_0\in G$ we have:
$$
A\in \Psi^m(G)
\Longrightarrow 
{}_{x_0}A \ \mbox{and}\ A_{x_0}\ \mbox{are in} \ \Psi^m(G).
$$
Furthermore, we can integrate with respect to $x_0\in G$:
$$
A\in \Psi^m(G)
\Longrightarrow 
\int_G {}_{x_0}A \, dx_0 \ \mbox{and}\ \int_G A_{x_0} dx_0\ \mbox{are in} \ \Psi^m(G).
$$

Let us prove a similar result for the classical calculus:
\begin{proposition}
\label{prop_transPsiclG}
For any $A\in \Psi^m_{cl}(G)$	and $x_0\in G$, 
the operators ${}_{x_0}A$, $A_{x_0}$, $\int_G {}_{x_0}A \, dx_0$ and $\int_G A_{x_0} dx_0$ are in $\Psi^m_{cl}(G)$.
\end{proposition}
\begin{proof}
We can construct a finite open cover $(\Omega_{j})_{j}$ of $G$ and 
a subordinate partition of unity of the form $\chi_j=\psi_j (\cdot \, z_j)$
for distinct point $z_j$, 
and $\psi_j$ valued in $[0,1]$, identically 1 near $e_G$ and with a small support about $e_G$. 
The exponential mapping
$\exp_G$ is a smooth diffeomorphism from a neighbourhood  of 0 onto a neighbourhood  of $\supp \psi_j$.
For any $A:\cD(G)\to \cD'(G)$, 
we define
$A_{j,k}:\cD(\Omega)\to \cD'(\Omega)$
via 
$$
A_{j,k}(f)\, (x)=\left( \chi_{k }\, A\chi_{j}\left(f ( \cdot\, x_j)\right)\right)\, (x x_k^{-1}).
$$
Then $A:\cD(G)\to \cD'(G)$ is in $\Psi^m(G)$ if and only if all the operators $\phi \mapsto (A_{j,k}(\phi\circ\exp_G^{-1}))\circ \exp_G$ are in $\Psi^m(\cO)$.
And $A\in \Psi^m(G)$ is classical if and only if all the operators $\phi \mapsto (A_{j,k}(\phi\circ \exp_G^{-1}))\circ \exp_G$ are in $\Psi^m_{cl}(\cO)$.

We observe that $({}_{x_0}A)_{j,k}$ coincides with 
$A'_{j,k}$ constructed in a similar fashion, but with the $z_j$ being replaced by $x_0^{-1} z_j$.
The membership of ${}_{x_0}A$ follows readily.
Furthermore, all the operators $({}_{x_0}A)_{j,k}$
have an integral kernel which depends smoothly on $x_0$,
i.e. $G\ni x_0\mapsto K_{ ({}_{x_0}A)_{j,k}} \in \cD'(G\times G)$ is continuous;
therefore the Euclidean symbols of $\phi \mapsto (({}_{x_0}A)_{j,k}(\phi\circ \exp_G^{-1}))\circ \exp_G$ depend smoothly on $x_0$ as well.
Hence, one checks easily that the integration over $G$ also produces a classical symbols.
This implies the membership of $\int_G {}_{x_0}A \, dx_0$ in $\Psi^m_{cl}(G)$.

Similarly, we obtain the $\Psi^m_{cl}(G)$-memberships of $A_{x_0}$ by considering left translated of $\Omega$
and of $\int_G A_{x_0} dx_0$ as above.
 \end{proof}

The proof of Proposition \ref{prop_transPsiclG}
yields readily:

\begin{corollary}
\label{cor_prop_transPsiclG}
Let $A\in \Psi^m_{cl}(G)$	and let $x_0\in G$.

\begin{enumerate}
\item
Denoting by $a_m$ the principal geometric symbol of $A$,
the principal geometric symbols of ${}_{x_0}A$ and $A_{x_0}$ at $x$
are given by $a_{m}(x_0x, L_{x_0}^*\cdot )$ and $a_{m}(xx_0,R_{x_0}^*\cdot)$ 
where $L_{x_0}^*$ and $R_{x_0}^*$
are the pullbacks of the left and right $x_0$-translation mappings.
\item
Translating the operator yields translation of the residue density
$$
\res_{x_1}({}_{x_0}A)
=
\res_{x_0x_1}(A)
\qquad\mbox{and}\qquad
\res_{x_1}(A_{x_0})
=
\res_{x_1x_0}(A),
$$
and of the canonical trace density:
$$
\TR_{x_1}({}_{x_0}A)
=
\TR_{x_0x_1}(A)
\qquad\mbox{and}\qquad
\TR_{x_1}(A_{x_0})
=
\TR_{x_1x_0}(A),
$$
Consequently, the operators $A$, ${}_{x_0}A$, $A_{x_0}$, $\int_G {}_{x_0}A \, dx_0$ and $\int_G A_{x_0} dx_0$
have the same non-commutative residue
and the same canonical trace.
If $m<-n$, they 
have the same trace.
\end{enumerate}
\end{corollary}

In the next section, we will use Part (1) to identify the principal symbol of a classical operator with a function $G\times (\fg^*\backslash\{0\})\to \bC$
which is homogeneous in $\xi\in \fg^*\backslash\{0\}$.
We will not use Part (2) in this paper.

\subsection{Principal symbols}
\label{subsec_princ_symbol}

Our definition of homogeneous symbol is motivated by the following important property
for which we  need the following conventions.
Recall that 
for a non-trivial (unitary) representation $\pi$
 when a maximal torus $T$ is fixed, 
 the representation space $\cH_\pi$
 decomposes orthogonally into  $\pi(T)$-eigenspaces.
 The  corresponding non-zero eigenvalues form the set of weights.
 We will also use the notion of analytical integral weight.
Some well-known facts will also recalled and used be below.
References for this classical material include  
 \cite{knapp_bk} (especially Chapters IV and V) and \cite{hall_bk} (especially Section 12).

\begin{lemma}
\label{lem_homG}
	Let $A=\Op_G(\sigma_A)\in \Psi^0_{cl}(G)$.
	Let $\chi_1,\chi_2\in \cD(G)$ with small supports near $e_G$ and identically equal to 1 on a small neighbourhood of $e_G$.
Let $b \in S^0(\bR^n)$ be the principal symbol of the operator $\exp ^* \chi_1 A\chi_2:f \mapsto (\chi_1 A((\chi_2f)\circ \exp_G))\circ \exp_G^{-1}$.
Let $w\in \fg^*$.
If there exists a maximal torus $T$ of $G$ and a non-trivial irreducible representation $\pi$ of $G$ such that $w$ is an analytical integral weight for $\pi$ and $T$ then 
$$
b(0,w) =\lim_{k\to +\infty} (\sigma_A(e_G,\pi^{\otimes k}) 
v^{\otimes k}, 
v^{\otimes k})_{\cH_\pi^{\otimes k}},
$$
where $v$ is a  unit $w$-weight vector.
\end{lemma}

As the eigenspace of an analytical integral weight is one-dimensional,
 the limit relation in Lemma \ref{lem_homG} does not depend on the choice of a unit highest weight vector. 

  

\begin{proof}[Proof of Lemma \ref{lem_homG}]
If $f$ is a smooth function supported in a neighbourhood small enough of $e_G$ and if $X$ is in a neighbourhood small enough of $e_G$, we have:
\begin{align*}
A(f\circ \exp_G^{-1})\ (\exp_G(X))
&=
\int_G f\circ \exp_G^{-1} (y) \kappa_{A,\exp_G(X)}(y^{-1} \exp_G(X)) dy
\\
&=	
\int_{\fg} f(Y) \kappa_{A,\exp_G(X)}(\exp_G(-Y)\exp_G(X)) 
|\jac_Y \exp|\ dY
\end{align*}
after the change of variable $y=\exp_G(Y)$.
Hence
the symbol $a$ of the operator $\exp ^* \chi_1 A\chi_2$
 is, up to a smoothing symbol,  given by
$$
a(X,\xi)= \int_{Y\in \fg} \chi(Y)\kappa_{A,\exp_G(X)} (\exp_G(Y-X)\exp_G(X)) e^{i\xi \cdot Y}
|\jac_Y \exp| \
dY,
$$
where $\chi\in \cD(\bR)$ is valued in $[0,1]$ and supported in a small neighbourhood of 0 
with $\chi\equiv 1$  near 0.
Consequently, its principal symbol at $X=0$ is 
$$
b(0,\xi)=\lim_{r\to +\infty} 	
	\int_{Y\in \fg} \chi(Y) \kappa_{A,e_G} (\exp_G(Y)) e^{ir\xi \cdot Y}
|\jac_Y \exp| \
	dY.
$$

Let $w\in \fg^*$ such that there exists a maximal torus $T$ of $G$ and a non-trivial irreducible representation $\pi$ of $G$ such that $w$
 is an analytical integral weight for $\pi$ and $T$.
We fix an ordering on $\fg^*$ so that $w$ is dominant and has become the highest weight of $\pi$.
Let  $v$ be a unit highest weight vector.
For every $k\in \bN$,
the highest weight theory yields:
$$
\forall y\in G\qquad
(\pi_{kw}(y)v_{kw}, 
v_{kw})_{\cH_{\pi_{kw}}}
=
(\pi^{\otimes k}(y)v^{\otimes k}, 
v^{\otimes k})_{\cH_\pi^{\otimes k}}
$$
where $\pi_{k w}$ is the representation with highest weight $kw$ and $v_{kw}$ a unit highest weight vector.
Therefore, we have for any $x\in G$
\begin{align}
(\sigma_A(x,\pi^{\otimes k}) 
v^{\otimes k}, 
v^{\otimes k})_{\cH_{\pi^{\otimes k}}}
&=
(\sigma_A(x,\pi_{kw}) 
v_{kw}, 
v_{kw})_{\cH_{\pi_{kw}}}\label{eq_sigmapootimesk}
\\
&=
(\widehat \kappa_{A,x} 
(\pi_{kw}) v_{kw}, v_{kw})_{\cH_{\pi_{k_w}}}.	\nonumber
\end{align}
Setting $\chi_G=\chi\circ \exp$, we have 
$\widehat \kappa_{A,e_G} = \cF_G (\chi_G\kappa_{A,e_G})
+  \cF_G ((1-\chi_G)\kappa_{A,e_G})$, 
and the second term is smoothing, so 
$$
|(\cF_G \left((1-\chi_G)\kappa_{A,e_G} \right)(\pi_{kw}) v_{kw}, 
v_{kw})_{\cH_{\pi_{k w}}}|
\leq
\|\cF_G \left((1-\chi_G)\kappa_{A,e_G} \right)(\pi_{kw})\|_{\sL(\cH_{\pi_{k w}})}
 $$
for any $N\in \bN$.
By \cite[Lemma 3.8]{monarxiv}, 
$\langle \lambda_{\pi_{kw}}\rangle\asymp
\langle kw \rangle^2\asymp k^2$ for $w$ fixed.
For the first term, we use 
$$
(\pi_{kw}(e^Y)v_{kw}, 
v_{kw})_{\cH_{\pi_{kw}}}
=
 e^{ikw(Y)},
$$
for $Y$ in a small neighbourhood of 0
to obtain:
$$
(\cF_G \left(\chi_G\kappa_{A,e_G} \right)(\pi_{kw}) v_{kw}, 
v_{kw})_{\cH_{\pi_{k w}}}
=
\int_\fg
\chi(Y)\kappa_{A, e_G}(\exp_G(Y))
e^{k i w(Y)}
|\jac_Y \exp| \
	dY
$$
This implies easily the limit relation in the statement.
\end{proof}

Having chosen a maximal torus $T$ in $G$,
the set of analytical integral functionals 
$$
\{w \ \mbox{analytical integral weight for}\ \pi\ \mbox{and}\ T \ :  \ \pi\in \Gh\backslash\{1_{\Gh}\}\},
$$
form a lattice of the dual $\ft^*$ of the Lie algebra $\ft$ of $T$
(see also \cite[Section 2.4]{monarxiv} for a discussion with this viewpoint).
It is easily seen that
given any real finite-dimension vector space $V$, 
 a continuous homogeneous function on $V\backslash\{0\}$  is characterised by its restriction to any lattice of $V$.
Hence, Lemma \ref{lem_homG} implies  that, 
keeping its notation, 
$b(0,\cdot)$ is determined completely on $\ft^*\backslash\{0\}$.
Now the unions of all the possible maximal tori and of the duals of their Lie algebras are $G$ and $\fg^*$ respectively.
Therefore, 
$b$ is determined completely on $\{0\}\times \fg^*\backslash\{0\}$.
These considerations together with 
 the properties of the translations (see Section \ref{subsec_invariance})
 and the definition of principal symbols yield:
 
 \begin{theorem}
\label{thm_princ_symb}
Let $A=\Op_G(\sigma_A)\in \Psi^0_{cl}(G)$.
Using the homogeneity and the pullback of the left translations, 
its principal symbol  is identified with a smooth function $a_0 :G\times (\fg^*\backslash\{0\})\to \bC$ which is  homogeneous of degree 0 in the variable $\xi\in \fg^*\backslash\{0\}$. 
Then we have for $x\in G$ and any non-trivial $\pi\in \Gh$, 
$$
a_0(x,w) =\lim_{k\to +\infty} (\sigma_A(x,\pi^{\otimes k}) 
v^{\otimes k}, 
v^{\otimes k})_{\cH_\pi^{\otimes k}},
$$
where $w$ and $v\in \cH_\pi$ are the highest weight and a corresponding unit vector for a given maximal torus $T$ and ordering on $\fg^*$.
This characterises $a_0$.
\end{theorem}

\subsection{Homogeneous symbols and classical symbols on $G$}
\label{subsec_hom_symb}

We are led to define the following notion of homogeneity.

\begin{definition}
\label{def_0homsymbolG}
Let $\sigma$ be a symbol.
It is said to be \emph{0-homogeneous} when the following condition is satisfied
for every $x\in G$ and every  $\pi\in \Gh\backslash\{1\}$
of $G$.
We fix a maximal torus $T$ and a realisation of $\pi$ as a unitary irreducible representation of $G$. We consider an orthonormal basis of $\pi(T)$-eigenvectors for the representation space.
The only possibly non-zero coefficients of $\sigma(x,\pi)$ 
are the diagonal ones corresponding to analytical integral weights, 
i.e. 
 $(\sigma(x,\pi)v_1,v_2)=0$ if $v_1,v_2$ are $\pi(T)$-eigenvectors 
 for weights which are not both analytical integral, 
 or if $v_1,v_2$ are $\pi(T)$-eigenvectors 
 for distinct analytical integral weights.
 Furthermore, 
 if $v$ is a unit eigenvector for an analytical integral weight for $\pi$ and $T$,
then we have for every $k\in \bN$
$$
(\sigma(x,\pi) v,v)_{\cH_\pi}= 
(\sigma(x,\pi^{\otimes k}) v^{\otimes k},v^{\otimes k})_{\cH_{\pi^{\otimes k}}}
$$
\end{definition}
 
In the last relation,  $\pi^{\otimes k}$ denotes the $k$-tensor product of the representation $\pi$.
As $v$ is a unit eigenvector for an analytical integral weight $w$ for $\pi$ and $T$,
 $v^{\otimes k}$ is a unit eigenvector for the analytical integral weight $kw$ for $\pi^{\otimes k}$ and $T$.
Furthermore, this last relation does not depend on the choice of the unit eigenvector $v$ since the eigenspace of an analytical integral weight 
 is one-dimensional.
 In fact, it implies
 $$
(\sigma(x,\pi) v,v)_{\cH_\pi}= 
(\sigma(x,\pi_{kw}) v_{kw},v_{kw})_{\cH_{\pi_{wk}}},
$$
where $\pi_{k w}$ is the representation with highest weight $kw$ and $v_{kw}$ a unit highest weight vector, see \eqref{eq_sigmapootimesk}.

Definition \ref{def_0homsymbolG}
gives a condition on the symbol $\sigma$ 
 as a field over $G\times \Gh$.
 Indeed,  
 if $\pi_1$ is equivalent to the representation $\pi$ in the definition, 
 then $\pi_1$ and $\pi$ are isomorphic and therefore have the same matrix representation in an orthonormal basis of weight vectors.
Furthermore, the condition in  Definition \ref{def_0homsymbolG} does not depend on the choice of a maximal torus $T$.
Indeed, if $T'$ is another maximal torus $T$ 
then $T'$ and $T$ are conjugate by an element $g$ of $G$;
  the co-adjoint action at $g$ is an isomorphism for the two sets of weights and the unitary map $\pi(g)$  intertwines the matrix representation in the basis of eigenvectors for $\pi(T)$ and $\pi(T')$.

\medskip

We can now define our notion of classical poly-homogeneous expansion for our global symbol in $G$:

\begin{definition}
A symbol $\sigma $ is said to be \emph{$m$-homogeneous} with $m\in \bC$
when $\lambda_\pi^{m/2}\sigma(x,\pi)$ yields a 0-homogeneous symbol.

A symbol is \emph{classical of complex order} $m\in \bC$ when 
 $\sigma$ is in $S^{\Re m}(G)$
 and admits the expansion
$\sigma\sim_h \sum_{j\in \bN_0} \sigma_{m-j}$ 
where each $\sigma_{m-j}$ is in $S^{\Re m-j}$ and is $(m-j)$-homogeneous.
\end{definition}

 In the case of the group being a torus $G=\bT^n$, 
 this definition is consistent with the ones given in Section \ref{subsec_psiclT}. 
 Therefore, we can generalise the notation of that section
 and  denote by $S^m_{cl}(G)$ the space of classical symbol of order $m\in \bC$.
We obtain as in the case of the torus:

\begin{theorem}
\label{thm:PsiclG}
For any $m\in \bC$, we have
$$\Psi^m_{cl}(G)=\Op_{G}(S^m_{cl}(G)).$$
\end{theorem}

The next section is devoted to the proof of Theorem \ref{thm:PsiclG}.

\subsection{Proof of Theorem \ref{thm:PsiclG}}

Let us first show the inclusion 
$\Op_{G}(S^m_{cl}(G))\subset \Psi^m_{cl}(G)$.
By  linearity of $\Op_{G}$ and the result on $\cup_{m_1\in \bR} \Psi^{m_1}(G)$, it suffices to show:
\begin{lemma}
\label{lem_Opsigma_cl}
If  $\sigma\in S^{\Re m}$ is $m$-homogeneous, then $\Op_G(\sigma)\in \Psi^m_{cl}(G)$.
\end{lemma}

\begin{proof}[Proof of Lemma \ref{lem_Opsigma_cl}]
We already know  $A:=\Op_G(\sigma)\in \Psi^{\Re m}(G)$.
If $\chi\in \cD(\bR)$ has a small enough support near 0 with $\chi(0)=1$, 
then the operator spectrally defined via $\chi(\cL)$ is the projection onto the constant function and $\cL+\chi(\cL)$ is injective.
 By \cite[Proposition 3.14]{monJFA}, 
the operator $\chi(\cL)$ is   smoothing  and 
$ (\cL +\chi(\cL))^{-m_1/2} \in \Psi^{\Re m_1}(G)$ for any $m_1\in \bC$.
Furthermore, the properties of a Laplace operator on a compact manifold 
\cite{seeley_67} imply $ (\cL +\chi(\cL))^{-m_1/2} \in \Psi^{m_1}_{cl}(G)$.
The properties of $\Op_G$ yield
$A_1:= A(\cL +\chi(\cL))^{-m/2}:=\Op_G (\sigma_1)$
where $\sigma_1$ is the symbol given by 
$$
\sigma_1(x,1_{\Gh}):=\sigma(x,1_{\Gh})
\qquad\mbox{and}\qquad 
\sigma_1(x,\pi)=\sigma(x,\pi) \lambda_\pi^{-m/2},
\quad \pi\in \Gh\backslash\{1_{\Gh}\}.
$$	
Showing that $A_1$ is classical will imply that $A=A_1(\cL +\chi(\cL))^{m/2}$ is also classical, 
therefore this paragraph shows that we may assume $m=0$.

Let us assume $m=0$. By the properties of invariance under translation
(see Section \ref{subsec_invariance}), it suffices to show that 
the operator $\exp ^* \chi_1 A\chi_2$ is classical where 
 $\chi_1,\chi_2\in \cD(G)$ have small supports near $e_G$ and identically equal to 1 on a small neighbourhood of $e_G$.
Proceeding as in the proof of Lemma \ref{lem_homG},
the symbol $a$ of the operator $\exp ^* \chi_1 A\chi_2$
 is, up to a smoothing symbol,  given by
$$
a(X,\xi)= \int_{Y\in \fg} \chi(Y)\kappa_{A,\exp_G(X)} (\exp_G(Y-X)\exp_G(X)) e^{i\xi \cdot Y}
|\jac_Y \exp| \
dY,
$$
where $\chi\in \cD(\fg^*)$ is valued in $[0,1]$, supported near $0$ and identically equal to 1 on a small neighbourhood of $0$.
For $X\in \supp \chi$, the function $\Phi_X$ defined via 
$$
\exp_G (\Phi_X (Y)) = \exp_G(Y-X)\exp_G(X), \qquad Y\in \fg^*,
$$
is smooth on $\fg^*$.
As $D_0\Phi_X =\id$, it is a diffeomorphism between open neighbourhoods of $0$.
The change of variable $Y'=\Phi_X (Y)$ yields
$$
a(X,\xi)= b_X(\xi\circ \Phi_X^{-1})
,
\quad\mbox{where}\quad
b^{(X)}(\xi):=
\int_{Y'\in \fg} \chi(\Phi_X^{-1}(Y'))
\kappa_{A,\exp_G(X)} (Y') e^{i\xi( Y')}
|\jac_{Y'} \exp| \
dY',
$$
viewing $\xi$ as an element of $\fg^*$.
Setting 
 $$
 b_T^{(X)}(w):= (\sigma(e^X,\pi) v,v),
$$
when  $w\in \fg^*$ is an analytical integral weight for some $\pi\in \Gh$ and maximal torus $T$, and 
where $v$ is a unit $w$-weight vector,
we have
$$
b^{(X)}(w) -b_T^{(X)}(w)
=
\cF_G \{(1-\chi_X) \ \kappa_{A,e^X}\}(\pi_w),
$$
where $\chi_X=\chi \circ \Phi_X^{-1} \circ \exp_G^{-1}\in \cD(G)$ is identically 1 in a neighbourhood of $e_G$.
Setting $b_T^{(X)}(0)=0$, the function
  $b^{(X)} -b_T^{(X)}$ is therefore an invariant smoothing symbol 
on the lattice of analytical integral weight for the maximal torus $T$.
Since $\sigma$ is 0-homogeneous, we have $b_T^{(X)}(kw)=b_T^{(X)}(w)$
and the function   $b_T^{(X)}$ is a 0-homogeneous symbol on the lattice of analytical integral $T$-weights.
By Proposition \ref{prop_ext_sigma_hom}, $b_T^{(X)}$ admits a unique smooth extension to $\ft^*\backslash\{0\}$ for which we keep the same notation.
The uniqueness of the construction shows that 
we can define a function $b_0^{(X)}:\fg^*\to \bC$ 
such that its restriction to the dual $\ft^*$ of the Lie algebra of any maximal torus $T$ coincides with $b_0^{(X)}|_ {\ft^*} =b_T^{(X)}$.
Furthermore,  $b_0^{(X)}\in C^\infty (\fg^*\backslash\{0\})$ is 0-homogeneous.

We fix a function $\psi_0\in C^\infty(\bR)$ be such that 
$\psi_0(s)=0$ for $s\leq1/2$ and 
$\psi_0(s)=1$ for $s\geq1$.
We define  the smooth and bounded function $R$ on $\fg^*$ with
$$
R_X(\xi):=b^{(X)}(\xi)-b_0^{(X)}(\xi)\psi(\xi),\qquad
\psi(\xi):=
\psi_0(|\xi|), \quad \xi\in \fg^*.
$$
By construction, $R_X$ is an invariant symbol in $S^0$ on $\bR^n\sim \fg^*$, therefore $\cF^{-1}_{\bR^n} R_X$ is Schwartz away from 0.
For any maximal torus $T$, the restriction of $R_X$ to the lattice of analytical integral weights of $T$ is smoothing, 
therefore it yields a smoothing symbol on the torus $T$ and its convolution kernel
$\cF^{-1}_{T} R_X|_T$ is smooth  near $e_G\in T$.
This being true for any maximal torus $T$ implies that $\cF^{-1}_{\bR^n} R_X$ is Schwartz on $\fg$, so $R_X\in \cS(\fg^*)$.
One checks easily that the map $X\mapsto R_X \in \cS(\fg^*)$ is  smooth on a neighbourhood of 0.
As the symbol $a$ is given by
$$
a(X,\xi) = 
b_0^{(X)}(\xi\circ \Psi_X^{-1})\psi(\xi\circ \Psi_X^{-1})
\ + \ R_X(\xi\circ \Psi_X^{-1}),
$$
it satisfies $a(X,\xi)\sim_h b_0^{(X)}(\xi\circ \Psi_X^{-1})+0$.
This concludes the proof.
\end{proof}

The main ingredient  for the reverse inclusion is the following lemma:

\begin{lemma}
\label{lem_Asigma_0}
Let $A\in \Psi^0_{cl}(G)$.	
As in Theorem \ref{thm_princ_symb},
we identify its principal symbol with 
a smooth function $a_0 :G\times (\fg^*\backslash\{0\})\to \bC$
which is $0$-homogeneous in the variable $\xi\in \fg^*\backslash\{0\}$.
 
Let $x\in G$ and let  $\pi$ be an irreducible non-trivial representation of $G$. 
Fixing a maximal torus $T$, 
we define an endomorphism  $M_{x,\pi}$ of $\cH_\pi$ in the following way:
all its coefficients with respect to  an orthonormal basis of $\pi(T)$-eigenvectors of $\cH_\pi$
 vanish except (potentially) 
the diagonal ones corresponding to 
analytical integral weights $w$
where we have
$(M_{x,\pi} v,v)_{\cH_\pi} := a_0(x,w)$.
When $x$ runs over $G$ and $\pi$ runs over the set of irreducible non-trivial representation of $G$, 
the matrices $M_{x,\pi}$ together with $\sigma_0(x,1_{\Gh})=0$ define a 0-homogeneous symbol $\sigma_0$.
Furthermore, $\sigma_0\in S^0(G)$ and 
$A-\Op(\sigma_0)\in \Psi^{-1}_{cl}(G)$. 
\end{lemma}

\begin{proof}[Proof of Lemma \ref{lem_Asigma_0}]
The considerations in Sections \ref{subsec_princ_symbol} and \ref{subsec_hom_symb}
 show that the symbol $\sigma_0$ is well-defined and $0$-homogeneous.
Let us
 show $\sigma_0\in S^0(G)$.
We fix a maximal torus $T$. We will need the following observation:
 the highest weight theory implies  that 
if $w_1,\ldots, w_j$ are $j$ analytical integral weights for representations
$\pi_1,\ldots, \pi_j \in \Gh$, then 
$w_1+\ldots +w_j$ is an analytical integral weights for the irreducible representation with highest weight (in any choice of ordering) in the decomposition of $\pi_1\otimes \ldots \otimes \pi_j$ into irreducibles, so we have:
\begin{equation}
\label{eq_a0tensors}
a_0(x,w_1+\ldots +w_j) = \left(\sigma_0(x,\pi_1\otimes \ldots\otimes \pi_j) v_1\otimes\ldots\otimes v_j,v_1\otimes \ldots\otimes v_j\right)_{\cH_{\pi_1}\otimes\ldots\otimes  \cH_{\pi_j}},
\end{equation}
where each vector $v_k$ is a unitary $w_k$-vector.

Let $\varphi,\pi\in \Gh$.
The endomorphsims $\sigma_0(x,1_\varphi\otimes \pi)$ and $\sigma_0(x,\varphi\otimes \pi)$ of  $\cH_\pi\otimes \cH_\varphi$ are diagonal when viewed in an orthogonal basis of vectors of the form $v_\pi\otimes v_\varphi$ where $v_\pi$ and $v_\varphi$ are  eigenvectors for $\pi(\ft)$ and $\varphi(\ft)$ respectively.
Furthermore,  only the diagonal entries corresponding to analytical integral $w_\pi$ and $w_\varphi$ may be non-zero.
Therefore this is also the case for  $\Delta_\varphi\sigma_0(x,\pi)$ and we have by \eqref{eq_a0tensors}
 $$
 \left(\Delta_\varphi\sigma_0(x,\pi)  v_\varphi\otimes v_\pi,
 v_\varphi\otimes v_\pi\right)
 = 
 a_0 (x, w_\varphi +w_\pi)- a_0(x,w_\pi).
$$
This last expression is equal to $D_\xi a_0(x,\cdot)(w_\pi)$  for some $\xi \in [w_\pi,w_\pi+w_\varphi]$
by the mean value theorem.
The homogeneity of $a_0$ then yield the estimate 
$$
 |\left(\Delta_\varphi\sigma_0(x,\pi)  v_\varphi\otimes v_\pi,
 v_\varphi\otimes v_\pi\right)|\leq
|w_{\varphi}| |\xi|^{-1} \sup_{|\xi'|=1} \|D_{\xi'} a_0(x,\cdot)\|
\lesssim_{a_0,\varphi} \langle w_\pi \rangle^{-1}.
$$
By \cite[Lemma 3.8]{monarxiv}, $\langle w_\pi \rangle \asymp \langle \lambda_\pi\rangle^{1/2}$.
Therefore, $\|\Delta_\varphi\sigma_0(x,\pi)\|_{\sL(\cH_{\pi\otimes\varphi})}\lesssim_{a_0,\varphi} \langle \lambda_\pi \rangle^{-1/2}$.

More generally, 
we compute recursively for any symbol $\sigma$ independent of $x$ and  $\varphi'_1,\ldots,\varphi'_J,\pi\in \RepG$: 
\begin{align}	
&\Delta_{\varphi'_1}\ldots \Delta_{\varphi'_J}\sigma(\pi)
=
\sigma(x,\otimes_{k=1}^J
\varphi'_k\otimes \pi)
-
\sum_{j=1}^J
\sigma( 1_{\varphi'_j} \otimes_{k\not=j} \varphi'_k \otimes \pi)
\label{eq_recDelta}
\\
&\qquad +
\sum_{1\leq j_1<j_2\leq J}\!\!\!\!\!\!
\sigma(1_{\varphi'_{j_1}} \otimes1_{\varphi'_{j_2}}\otimes_{k\not=j_1,j_2} \varphi'_k \otimes \pi)
+\ldots \nonumber
\\
&\qquad +
(-1)^{J-1}\!\!\!\!\!\!
\sum_{1\leq j_1<j_2\leq J}\!\!\!\!\!\!
\sigma( 
\varphi'_{j_1} \otimes \varphi'_{j_2}\otimes_{k\not=j_1,j_2} 1_{\varphi'_k} \otimes \pi)
+(-1)^J\sum_{j=1}^J
\sigma( \varphi'_j \otimes_{k\not=j} 1_{\varphi'_k} \otimes \pi).
\nonumber
\end{align}
We now assume $\varphi'_1,\ldots,\varphi'_J,\pi$ irreducible.
The endomorphsim $\Delta_{\varphi'_1}\ldots \Delta_{\varphi'_J} \sigma_0(x,\pi)$ is 
 diagonal when viewed 
 in an orthogonal basis of vectors which are the projections of
of  $v_\pi\otimes v_{\varphi'_1}\otimes \ldots \otimes v_{\varphi'_J} $ 
and only the diagonal entries corresponding to analytical integral weights  $w_\pi$ and $w_{\varphi_k}$ may be non-zero. 
The formulae in \eqref{eq_a0tensors} and \eqref{eq_recDelta} applied to symbols independent of $x$ on $\Gh$ and also on the lattice of analytical integral weights yield:
$$
 \left(\Delta_{\varphi'_1}\ldots \Delta_{\varphi'_J}\sigma_0(x,\pi) 
v_\pi\otimes v_{\varphi'_1}\otimes \ldots \otimes v_{\varphi'_J},
v_\pi\otimes v_{\varphi'_1}\otimes \ldots \otimes v_{\varphi'_J}\right)
= 
\Delta_{w_{\varphi'_1}}\ldots \Delta_{w_{\varphi'_J}}a_0(x,w_\pi) ,
$$
where  $a(x,\cdot)$ is restricted to the lattice of analytical integral weights of $G$.
Iteratively applying the mean value formula, we have
$$
\Delta_{w_{\varphi'_1}}\ldots \Delta_{w_{\varphi'_J}}a(x,w_\pi) 
=
D^J_{\xi'} a_0(x,\cdot)(w_{\varphi'_1},\ldots,w_{\varphi'_J})
$$
for some $\xi'$ in a ball about $w_\pi$ and with radius $\sum_{j=1}^J |w_{\phi'_j}|$.
The homogeneity of $a_0$ then yield
$$
|D^J_{\xi'} a_0(x,\cdot)(w_{\varphi'_1},\ldots,w_{\varphi'_J})|
\leq |\xi'|^{-J} |w_{\varphi'_1}|\ldots|w_{\varphi'_J}|
\sup_{|\xi'|=1} \|D^J_{\xi'} a_0(x,\cdot)\|
\lesssim \langle w_\pi\rangle^{-J}.
$$
As $\langle w_\pi \rangle \asymp \langle \lambda_\pi\rangle^{1/2}$, 
we have obtained
$\|\Delta_{\varphi'_1}\ldots \Delta_{\varphi'_J}\sigma_0(x,\pi)\|_{\sL}
\lesssim \langle \lambda_\pi\rangle^{-J/2}$.
This shows that the symbol $\sigma_0$ is in $S^0(G)$.

By Lemma \ref{thm:PsiclG}, $\Op(\sigma_0)\in\Psi^0_{cl}(G)$
so $A-\Op(\sigma_0)$ is in $\Psi^0_{cl}(G)$.
By construction and Theorem \ref{thm_princ_symb}, 
the principal symbol of $A-\Op(\sigma_0)$ is identically zero.
This implies $A-\Op(\sigma_0)\in\Psi^{-1}_{cl}(G)$
and concludes the proof of Lemma \ref{lem_Asigma_0}.
\end{proof}

\begin{proof}[Proof of Theorem \ref{thm:PsiclG}]
We have already noticed that Lemma \ref{lem_Opsigma_cl} implies
$\Op_{G}(S^m_{cl}(G))\subset \Psi^m_{cl}(G)$.
Let us prove the reverse inclusion.
Let $A\in \Psi^m_{cl}(G)$.
Then $A(\id+\cL)^{-\frac m2}\in \Psi^0_{cl}(G)$
and we denote by  $\sigma_0$  the $0$-homogeneous symbol
associated with it
in Lemma \ref{lem_Asigma_0}.
We define $\sigma_m$ via $\sigma(x,\pi) = \lambda_\pi^{\frac m2} \sigma_m(x,\pi)$ for $\pi\in \Gh\backslash\{1_{\Gh}\}$, 
and $\sigma(x,1_\Gh) = 0$. Then the symbol $\sigma_m$ is in $S^m(G)$ 
while the operator 
$$
A-\Op(\sigma_m)
=
\left(A(\id+\cL)^{-\frac m2} -\Op(\sigma_0)\right) (\id+\cL)^{\frac m2}
\ \mbox{mod}\Psi^{-\infty},
$$
is in $\Psi^{m-1}_{cl}(G)$.
Recursively, we obtain a poly-homogeneous expansion $\sum_{j\in \bN_0}\sigma_{m-j}$ for $\sigma=\Op_G^{-1} A$
and this shows  $\sigma\in S^m_{cl}(G)$.
We have obtained the reverse inclusion 
$\Psi^m_{cl}(G)\subset \Op_{\bT^n}(S^m_{cl}(G))$
and this concludes the proof of Theorem \ref{thm:PsiclG}.
\end{proof}

\subsection{Non-commutative residue and canonical trace on $G$}

Here, we generalise Corollary \ref{cor_TR_thm:PsiclT} to the case of a compact Lie group which may not be commutative, i.e. may not be the torus.
Applying Theorem \ref{thm_myexp} and Proposition 
\ref{prop_TR_in_exp} to the Laplace operator $\cL$,
 we obtain readily the following result.

\begin{proposition}
		Let  $\sigma\in S^m_{cl}(G)$ with $m\in \bC$. 
We set $A:=\Op_{G}(G)$.

\begin{enumerate}
\item Let $\eta\in \cD(0,\infty)$.
Then we have an expansion as $t\to 0^+$:
$$
\tr(A\, \eta(\cL)) = 
\sum_{\pi\in \Gh} d_\pi
\eta(t\lambda_\pi) \tr (\sigma(x,\pi))
\sim 
c_{m+n}
t^{-\frac{m+n}{m_0}} 
+
c_{m-n-1}
t^{-\frac{m+n-1}{m_0}} 
+
\ldots 
$$	
The constants $c_{m+n-j}$  are of the form 
$c_{m+n-j} = c_{m+n-j}^{(\sigma)} c_{m+n-j}^{(\eta)}$
where $c_{m+n-j}^{(\sigma)}$ depends only on the homogeneous expansion of $\sigma$ and 
$$
c_{m+n-j}^{(\eta)}:=
\frac12
\int_{u=0}^{+\infty}
\eta(u) \ u^{\frac{m-j+n}{2}} \frac{du}u .
$$	
If  $m\in \bZ_n$ then $c_0=c_0^{(\sigma)}c_0^{(\eta)}$ with
$$
c_0^{(\sigma)}=\res A,\qquad\mbox{and}\qquad
c_0^{(\eta)}=\frac 12 \int_0^{+\infty} \!\!\! \eta(u)\ \frac{du}u .
$$

\item 
Here, $\Re m\geq -n$ with $m \notin \bZ$
and $N\in \bN$ is  such that $\Re m+n<N$.
Let $\eta\in \cD(\bR)$ 	such that  $\eta\equiv 1$ on a neighbourhood of 
$0$.
Then we have an expansion as $t\to 0^+$:
$$
\tr(A\, \eta(\cL)) = 
\sum_{\pi\in \Gh} d_\pi
\eta(t\lambda_\pi) \tr (\sigma(x,\pi))
= 
\TR(A) + \sum_{j=0}^{N-1} c'_{m+n-j} t^{\frac{-m-n+j}{m_0}}  +o(1).
$$
The constants $c'_{m+n-j}$  are of the form 
$c'_{m+n-j} = {c'}_{m+n-j}^{(\sigma)} {c'}_{m+n-j}^{(\eta)}$
where ${c'}_{m+n-j}^{(\sigma)}$ depends only on the poly-homogeneous expansion of $\sigma$ and 
$$
{c'}_{m+n-j}^{(\eta)}:=
\frac1{2}
\int_{u=0}^{+\infty}
\eta(u) \ u^{\frac{m-j+n}{2}} \frac{du}u .
$$	
\end{enumerate}
\end{proposition}

\end{document}